\documentclass[twocolumn]{autart}    

\usepackage{cite} 
\usepackage{color}
\usepackage{graphicx}          
\usepackage{amsmath}
\usepackage{amssymb}
\usepackage{amsfonts}
\usepackage{eucal}
\usepackage[titletoc,title]{appendix}
\usepackage{mathptmx} 
\usepackage{epsfig}
\usepackage{epstopdf}
\usepackage{comment}
\usepackage{tikz} \usetikzlibrary{calc} \tikzset{>=latex}
\usepackage{bm}
\usepackage{stfloats}
\usepackage{subfig}
\usepackage{bbm}
\usepackage{mathrsfs} 
\usepackage{float}
\usepackage{amsmath}
\usepackage{tikz}
	\usetikzlibrary{calc}
	\usetikzlibrary{arrows}
	\usetikzlibrary{positioning}
	\usetikzlibrary{decorations.markings}
\newtheorem{theorem}{Theorem}
\newtheorem{lemm}{Lemma}
\newtheorem{propo}{Proposition}
\newtheorem{remark}{Remark}
\newtheorem{definition}{Definition}
\newtheorem{assumption}{Assumption}
\newtheorem{corollary}{Corollary}
\newenvironment{proof}[1][Proof]{\noindent\textbf{#1.} }{\ \rule{0.5em}{0.5em}}

\begin{document}

        \setlength\abovedisplayskip{6pt}
        \setlength\belowdisplayskip{6pt}
        \setlength\abovedisplayshortskip{6pt}
        \setlength\belowdisplayshortskip{6pt}
        \allowdisplaybreaks
        \setlength{\parindent}{1em}
        \setlength{\parskip}{-0em}  
        \addtolength{\oddsidemargin}{5pt}      
        
        \begin{frontmatter}
                
                \title{Output Feedback  Periodic Event-Triggered and Self-Triggered Boundary Control of Coupled $2\times 2$ Linear Hyperbolic PDEs} 
    \thanks[t1]{A preliminary version of this work appears in the American Control Conference, ACC 2025. This work was funded by the NSF CAREER Award CMMI-2302030 and  the NSF grant CMMI-2222250.}

\thanks[footnoteinfo]{Corresponding author E.~Somathilake Tel. +1(858)-319-6429.}

\author[mae]{Eranda Somathilake\thanksref{footnoteinfo}}\ead{esomathilake@ucsd.edu},  \author[ece]{Bhathiya  Rathnayake,}\ead{brm222@ucsd.edu} \ and\  
\author[mae]{Mamadou Diagne}\ead{mdiagne@ucsd.edu }               

\address[mae]{Department of Mechanical and Aerospace Engineering, University of California, San Diego, La Jolla, CA 92093, USA}

\address[ece]{Department of Electrical and Computer Engineering, University of California, San Diego, La Jolla, CA 92093, USA}  


\begin{abstract}
In this paper, we expand recently introduced  observer-based periodic event-triggered  control (PETC) and self-triggered control (STC) schemes for reaction-diffusion PDEs to boundary control of $2\times2$ coupled hyperbolic PDEs in canonical form and with anti-collocated measurement and actuation processes. The class of problem under study governs transport phenomena arising in water management systems, oil drilling, and traffic flow, to name a few. Relative to the state of the art in observer-based event-triggered control of hyperbolic PDEs, our contribution goes two steps further by proposing observer-based PETC and STC for the considered class of systems. These designs arise from a non-trivial redesign of an existing continuous-time event-triggered control (CETC) scheme. PETC and STC eliminate the need for constant monitoring of an event-triggering function as in CETC; PETC requires only periodic evaluations of the triggering function for event detection, whereas STC is a predictor-feedback that anticipates the next event time at the current event exploiting continuously accessible output measurements. The introduced resource-aware designs act as input holding mechanisms allowing for the update of the input signal only at events. Subject to the designed boundary output feedback PETC and STC control laws characterized by a set of event-trigger design parameters, the resulting closed-loop systems, which are inherently Zeno-free by design, achieve  exponential convergence to zero in the spatial $L^2$ norm. We illustrate the feasibility of the approach by applying the control laws to the linearized Saint-Venant model, which describes the dynamics of shallow water waves in a canal and is used to design flow stabilizer feedback laws via gate actuation. The provided simulation results illustrate the proposed theory.  
\end{abstract}
                
      \end{frontmatter}

\section{Introduction}

Hyperbolic PDEs provide a mathematical description  of transport processes, elucidating the movement of fluids, biological systems, environmental systems, chemical engineering, and heat transfer, to name a few. Emanating from  first-principles, these PDEs are useful to the estimation, the prediction, and the control of various systems such as open channel fluid flow \cite{2003Boundary,di2012backstepping, diagne2017backstepping,diagne2017control,diagne2012lyapunov,strecker2022,bojan2022}, blood flow dynamics \cite{bekiaris20221},   traffic systems \cite{yu2019traffic},     operation of natural gas pipeline networks \cite{hari2021operation}, to cite a variety.
\subsection{Boundary control of hyperbolic PDEs}
Over the past few decades, the field has witnessed  major developments toward the conception of stabilizing boundary feedback laws for coupled linear hyperbolic systems.  An important  advancement occurred with the introduction of a controller that achieves local exponential stabilization of the Saint-Venant model \cite{coron1999lyapunov}. Subsequent contributions by \cite{bastin2010further, bastin2011boundary} , as well as \cite{Vazquez2011}, have played a key role in formulating a quadratic Lyapunov candidate specifically designed for $2\times 2$ linear hyperbolic systems in canonical form. The result in  \cite{Vazquez2011}, which pertains to PDE backstepping design, is extended to an $n+1$ system in \cite{Meglio2013}. The generalization of \cite{Vazquez2011} in \cite{Meglio2013}, led to the application of PDE backstepping design to a broad range of applications including 
$2+1$ systems such as  two-phase slugging processes in  oil drilling \cite{di2012backstepping}, drift-flux modeling in oil drilling \cite{aarsnes2014control}, and coupled water-sediment dynamics in river breaches \cite{diagne2017backstepping}. Backstepping design has been used to stabilize a $3+1$ multi-class traffic model  \cite{burkhardt2021stop} and the $2+2$ bi-layer Saint-Venant model  \cite{diagne2017control}. In light of the results in \cite{Meglio2013} and \cite{Hu2016}, the design of adaptive observers for general linear hyperbolic PDEs was proposed in \cite{Anfinsen2016} and \cite{anfinsen2017estimation}, respectively, and adaptive boundary control strategies in \cite{anfinsen2019adaptive}.  In the contributions reported above, global exponential stability is principally established with respect to the $L^2$ norm of the state, and both full-state feedback and observer-based control design are considered. Recently, leveraging DeepOnet—a Neural Operator (NO) approximation, \cite{wang2023neural}  uses Machine Learning to facilitate the  computation of gain kernel PDEs for coupled hyperbolic PDEs, and has even achieved Global Practical Exponential Stability when the output feedback law including the system's state is completely learned.  

\subsection{Event-based control of linear coupled hyperbolic systems}

The works mentioned above on hyperbolic PDEs rely on continuous-time control (CTC), which is often not feasible. A practical solution is sampled-data control, where the control input is updated according to a predetermined sampling schedule. However, the maximum allowable sampling interval of this schedule must be chosen conservatively to ensure that the control input is updated frequently enough. In a worst-case scenario, a high frequency of control updates is critical to prevent the closed-loop system from falling into an unstable manifold due to the open-loop operation between control updates. Alternatively, event-triggered control (ETC) introduces feedback into the control update tasks. Here, the control input is updated only when a certain condition related to the system states is met, which we call an \textit{event}. This feedback integration means the schedule of control update times is not restricted by a worst-case scenario. Control updates are triggered based on the current state rather than an unlikely worst-case scenario, leading to significantly sparser control updates compared to traditional sampled-data control methods.

The latest advancements in control systems theory have witnessed an unprecedented expansion of sampled-data \cite{fridman2012robust,karafyllis2018sampled,katz2022sampled, davo2018stability,karafyllis2017sampled,wang2022sampled} and event-based control applied to PDE systems. ETC for parabolic PDEs involving both static and dynamic triggering conditions can be found in  \cite{espitia2021event,katz2020boundary,rathnayake2021observer,rathnayake2022sampled,rathnayake2024observer,wang2022event}.  Dealing with event-based boundary control of hyperbolic systems, our current contributions align with works such as \cite{espitiaObserverbased2020,wang2022eventb,baudouin2023event,koudohode2022event,espitia2020event,espitiaTrafficFlowControl2022} with particular emphasis on \cite{espitiaObserverbased2020}, which proposes a globally exponentially convergent observer-based ETC for a $2\times 2$ hyperbolic system in the canonical form. Among these studies,  the authors of \cite{diagne2021event}  design a static triggering condition that achieves  exponential stabilization of  a scalar but nonlinear hyperbolic system describing parts flow in a highly re-entrant manufacturing plant. The model of concern is  often used to describe semi-conductors and chip assembly lines. The works \cite{espitia2020event,espitiaTrafficFlowControl2022} offer application of ETC designs to traffic systems. 

\subsection{Contributions}
This contribution presents  two major improvements to the state of the art of observer-based ETC for hyperbolic PDEs by removing the necessity of continuous monitoring of the triggering functions, an unavoidable step when implementing the regular ETC stabilizer (for instance \cite{espitiaObserverbased2020}), which we refer to as \emph{continuous-time} ETC (CETC). The PETC and STC designs introduced for a scalar reaction-diffusion PDE with explicit gain kernel functions \cite{rathnayake2023observer} are proven to apply to hyperbolic PDE systems of a higher level of complexity induced by coupled dynamics, leading to implicitly defined coupled gain kernel functions. The authors of \cite{peihan2024traffic} proposed \emph{full-state feedback} PETC and STC designs for a $2\times 2$ coupled hyperbolic traffic PDEs exploring the notion of a \textit{performance barrier} \cite{ong2023performance}, \cite{rathnayake2023prfmnce}. However, \emph{observer-based} PETC and STC designs for hyperbolic PDEs are still lacking, despite being crucial for practical control implementations in the context of flow and fluid systems for which distributed measurement of the state is often implausible. To address this significant gap, this work presents observer-based PETC and STC designs for $2\times 2$ coupled hyperbolic PDEs applied to the linearized Saint-Venant equations modeling water flow and level regulation problem via gate actuation. The PETC design arises from constructing a triggering function that requires only periodic evaluations. We achieve this by determining an upper bound for the CETC triggering function between two periodic evaluations. Simultaneously, we establish an upper bound for the permissible sampling period used for periodic evaluation of the triggering function. The STC design emerges from constructing a function that is positively lower bounded and depends on continuously available outputs. This function, when evaluated at an event time, determines the waiting time until the next event. We construct this function by determining bounds on several variables that appear in the CETC triggering function. 
Both PETC and STC are inherently free from Zeno behavior, and the well-posedness of the closed-loop system under both PETC and STC is provided. Furthermore, using Lyapunov arguments, we establish the exponential convergence of the closed-loop signals to zero in the $L^2$ spatial norm. The rapid development of  remote actuation systems for automated control of networks of irrigation canals, justifies the application example as preserving actuation and communication resources enables to expand the level of autonomy by drastically reducing energy and bandwidth consumption.

\textbf{\emph{Organization of the paper:}} 
Section~\ref{sec:problem_statement} presents an exponentially stabilizing boundary control law for a $2\times2$ linear hyperbolic system followed by its emulation for the event-triggering mechanisms as well as necessary preliminary results for the proposed control designs including the CETC design. Sections~\ref{sec:petc} and \ref{sec:stc} present the exponentially convergent PETC, and STC event-triggering mechanisms. Finally, we present the numerical simulations of the control strategies applied to the linearized Saint-Venant equations and concluding remarks in Sections~\ref{sec:model} and \ref{sec:conclusion}, respectively.

\textbf{\emph{Notation:}} Let $\mathbb{R}$ be the set of real numbers and $\mathbb{R}^+$ be the set of positive real numbers. Let $\mathbb{N}$ be the set of natural numbers including $0$. Define the constant $\ell \in \mathbb{R}^+$. By $L^2(0,\ell)$, we denote the equivalence class of Lebesgue measurable functions $f:[0,\ell]\rightarrow\mathbb{R}$ such that $\|f\|_{L^2((0,\ell);\mathbb{R})}=\left(\int_0^\ell|f(x)|^2\right)^{1/2}<\infty$. Define $\mathcal{C}^0(I;L^2((0,\ell);\mathbb{R}))$ as the space of continuous functions $u(\cdot,t)$ for an interval $I\subseteq\mathbb{R}^+$ such that $I\ni t\rightarrow u(\cdot,t)\in L^2((0,\ell);\mathbb{R})$. Also, for the equivalence class of Lebesgue measurable functions $\chi_1,\chi_2:[0,\ell]\rightarrow\mathbb{R}$, we define $\rVert\big(\chi_1,\chi_2\big)^T \rVert=\left(\|\chi_1\|_{L^2((0,\ell);\mathbb{R})}^2 + \|\chi_2\|_{L^2((0,\ell);\mathbb{R})}^2\right)^{1/2}$.
\maketitle


\section{\protect Preliminaries and problem formulation}\label{sec:problem_statement}

We consider the following $2\times 2$ linear hyperbolic  PDE system in the  canonical form where the independent variables $t\geq 0$ and $x\in[0,\ell]$ are time and space variables, respectively, and the PDE states $u(x,t)$ and $v(x,t)$ satisfy 
\begin{align}
\partial_t u(x,t) =&-\lambda_1 \partial_x u(x,t)+c_1(x) v(x,t), \label{sys_u} \\
\partial_t v(x,t)=&\lambda_2 \partial_x v(x,t)+ c_2(x) u(x,t),\label{sys_v}
\end{align} 
with boundary conditions 
\begin{align}
u(0,t) & = q v(0,t),\label{sys_BC1} \\
v(\ell,t) & =\rho u(\ell,t) + U(t).\label{sys_BC2}
\end{align}
Here, $U(t)$ is the continuous-time boundary control input, and $\lambda_1, \lambda_2\in \mathbb{R}^+$, $c_1(x),c_2(x)\in \mathcal{C}^0((0,\ell);\mathbb{R})$. Further, $q\neq0$ is the distal reflection term, and $\rho\neq0$ is the proximal reflection term. The initial conditions are chosen such that
\begin{align*}
(u^0, \; v^0)^T \in  L^2((0,\ell);\mathbb{R}^2).
\end{align*}
We make the following assumption on the reflection terms. 
\begin{assumption}\label{assum:reflection}\cite{espitiaObserverbased2020}
  The reflection terms are small enough such that the following inequality holds:
    \begin{align*}
        |\rho  q|\leq \frac{1}{2}.
    \end{align*}
\end{assumption}

Our goal here is to devise PETC and STC strategies to determine the event times at which the control input is updated. Prior to this, we first derive an observer-based continuous-time  boundary feedback law that is applied in a zero-order hold fashion between events determined by a continuous-time  triggering mechanism, leading to a CETC strategy. The design features an anti-collocated sensing and actuation setup, contrasting with the collocated setup considered in \cite{espitiaObserverbased2020}. Building upon the CETC design, we propose the PETC and STC strategies. The continuous-time  triggering function is converted into a periodic event-triggering function in the PETC design, and a mechanism to determine the next triggering time at the current triggering time is developed in the STC design.

\subsection{Continuous-time  output feedback control and its emulation for ETC}
In this subsection, we develop a continuous-time  backstepping output feedback control $U(t)$ capable of exponentially stabilizing the closed-loop system consisting of the plant \eqref{sys_u}-\eqref{sys_BC2} and an observer, using $v(0,t)$ as the available boundary measurement. Since the actuation and measurement are located at opposite boundaries, this setup is referred to as anti-collocated sensing and actuation, which differs from the configuration in \cite{espitiaObserverbased2020}, where both the actuation and measurement are located at the same boundary.

We design an  observer consisting of the copy of the plant plus some output injection terms with the observer states denoted by $(\hat{u},\hat{v})^T$. Defining the observer errors as 
\begin{align}\label{error_u}
&\tilde{u}:=u-\hat{u},\\\label{error_v}
&\tilde{v}:=v-\hat{v},
\end{align}
the following observer is proposed:
\begin{align}
\partial_t \hat{u}(x,t) =&-\lambda_1 \partial_x \hat{u}(x,t)+ c_1(x)\hat{v}(x,t)+p_1(x)\tilde{v}(0,t), \label{obs_u} \\
\partial_t \hat{v}(x,t)=&\lambda_2 \partial_x \hat{v}(x,t) + c_2(x)\hat{u}(x,t) + p_2(x)\tilde{v}(0,t),\label{obs_v}
\end{align}
with boundary conditions 
\begin{align}
\hat{u}(0,t) & = q v(0,t),\label{obs_BC1} \\
\hat{v}(\ell,t) & = \rho \hat{u}(\ell,t) + U(t),\label{obs_BC2}
\end{align}
and initial conditions such that
\begin{align*}
(\hat{u}^0, \; \hat{v}^0)^T \in  L^2((0,\ell);\mathbb{R}^2)
\end{align*}
The functions $p_1(x)$ and $p_2(x)$ are the observer output injection gains, which are to be determined through backstepping design to ensure the convergence of the estimated states to the plant states. It can be easily verified that the dynamics of the observer errors \eqref{error_u} and \eqref{error_v} satisfy
\begin{align}
\partial_t \tilde u(x,t) =&-\lambda_1 \partial_x \tilde{u}(x,t)+c_1(x)\tilde{v}(x,t) - p_1(x)\tilde{v}(0,t), \label{err_u} \\
\partial_t \tilde{v}(x,t)=&\lambda_2 \partial_x \tilde{v}(x,t)+ c_2(x)\tilde{u}(x,t) - p_{2}(x)\tilde{v}(0,t),\label{err_v}
\end{align}
with boundary conditions 
\begin{align}
\tilde u(0,t) & =0,\label{err_BC1} \\
\tilde v(\ell,t) & =\rho \tilde{u}(\ell,t),\label{err_BC2}
\end{align}
where the initial conditions satisfy
\begin{align*}
(\tilde{u}^0, \; \tilde{v}^0)^T \in  L^2((0,\ell);\mathbb{R}^2).
\end{align*}
Consider the following observer error backstepping transformations:
\begin{align}
\tilde{u}(x,t) & =\tilde{\alpha}(x,t)-\int_0^x P^{\alpha\alpha}(x, \xi) \tilde{\alpha}(\xi,t) d\xi\nonumber\\&\quad-\int_0^x P^{\alpha\beta}(x, \xi) \tilde{\beta}(\xi,t) d\xi,\label{err_backstepping1}\\
\tilde{v}(x,t) & =\tilde{\beta}(x,t)-\int_0^x P^{\beta\alpha}(x, \xi) \tilde{\alpha}(\xi,t) d\xi\nonumber\\&\quad-\int_0^x P^{\beta\beta}(x, \xi) \tilde{\beta}(\xi,t) d\xi,\label{err_backstepping2}
\end{align}
defined over the triangular domain $0 \leq \xi \leq x \leq \ell$, 
where the kernels $P^{\alpha\alpha},P^{\alpha\beta},P^{\beta\alpha},$ and $P^{\beta\beta}$ satisfy the following PDEs:
\begin{align}
\lambda_1\partial_xP^{\alpha\alpha}(x,\xi)+\lambda_1\partial_\xi P^{\alpha\alpha}(x,\xi)=&c_1(x)P^{\beta\alpha}(x,\xi)\label{p_alphaalpha},\\
\lambda_1\partial_xP^{\alpha\beta}(x,\xi)-\lambda_2\partial_\xi P^{\alpha\beta}(x,\xi)=&c_1(x)P^{\beta\beta}(x,\xi)\label{p_alphabeta},\\
\lambda_2\partial_xP^{\beta\alpha}(x,\xi)-\lambda_1\partial_\xi P^{\beta\alpha}(x,\xi)=&-c_2(x)P^{\alpha\alpha}(x,\xi)\label{p_betaalpha},\\
\lambda_2\partial_xP^{\beta\beta}(x,\xi)+\lambda_2\partial_\xi P^{\beta\beta}(x,\xi)=&-c_2(x)P^{\alpha\beta}(x,\xi)\label{p_betabeta},
\end{align}
with boundary conditions
\begin{align}
&P^{\alpha\alpha}(\ell,\xi)=\frac{1}{\rho}P^{\beta\alpha}(\ell,\xi),
&&P^{\alpha\beta}(x,x)=-\frac{c_1(x)}{\lambda_1+\lambda_2},\label{p_alphaalpha1__p_alphabetax}\\
&P^{\beta\beta}(\ell,\xi)=\rho P^{\alpha\beta}(\ell,\xi),
&&P^{\beta\alpha}(x,x)=\frac{c_2(x)}{\lambda_1+\lambda_2}.\label{p_betabeta1__p_betaalphax}
\end{align}
The output gain terms are chosen as
\begin{align*}
p_1(x)=&-\lambda_2P^{\alpha\beta}(x,0),\\
p_2(x)=&-\lambda_2P^{\beta\beta}(x,0).
\end{align*}          
Hence, the observer error system \eqref{err_u}-\eqref{err_BC2} is transformed into the following target system:
\begin{align}
\partial_t \tilde{\alpha}(x,t) =&-\lambda_1 \partial_x \tilde{\alpha}(x,t),\label{err_target_u} \\
\partial_t \tilde{\beta}(x,t)=&\lambda_2 \partial_x \tilde{\beta}(x,t),\label{err_target_v}
\end{align}
with boundary conditions 
\begin{align}
\tilde{\alpha}(0,t) & =0,\label{err_target_BC1} \\
\tilde{\beta}(\ell,t) & =\rho \tilde{\alpha}(\ell,t).\label{err_target_BC2}
\end{align}
The inverse transformations of \eqref{err_backstepping1},\eqref{err_backstepping2} are given by
\begin{align}
\tilde{\alpha}(x,t)=&\tilde{u}(x,t)+\int_0^xR^{uu}(x,\xi)\tilde{u}(\xi,t)d\xi\nonumber\\&+\int_0^xR^{uv}(x,\xi)\tilde{v}(\xi,t)d\xi,\label{err_backstepping1-inverse}\\
\tilde{\beta}(x,t)=&\tilde{v}(x,t)+\int_0^xR^{vu}(x,\xi)\tilde{u}(\xi,t)d\xi\nonumber\\&+\int_0^xR^{vv}(x,\xi)\tilde{v}(\xi,t)d\xi,\label{err_backstepping2-inverse}
\end{align}
defined over the triangular domain $0 \leq \xi \leq x \leq \ell$, where the kernels $R^{uu},R^{uv},R^{vu},$ and $R^{vv}$ satisfy the following PDEs:
\begin{align}
    \lambda_1 \partial_x R^{uu}(x,\xi) + \lambda_1 \partial_\xi R^{uu}(x,\xi)=&-c_2(\xi)R^{uv}(x,\xi),\label{r_uu}\\
    \lambda_1 \partial_x R^{uv}(x,\xi) - \lambda_2 \partial_\xi R^{uv}(x,\xi)=&-c_1(\xi)R^{uu}(x,\xi),\label{r_uv}\\
    \lambda_2 \partial_x R^{vu}(x,\xi) - \lambda_1 \partial_\xi R^{vu}(x,\xi)=&c_2(\xi)R^{vv}(x,\xi),\label{r_vu}\\
    \lambda_2 \partial_x R^{vv}(x,\xi) + \lambda_2 \partial_\xi R^{vv}(x,\xi)=&c_1(\xi)R^{vu}(x,\xi),\label{r_vv}
\end{align}
with boundary conditions
\begin{align}
    &R^{uu}(\ell,\xi)=\frac{1}{\rho}R^{vu}(\ell,\xi), 
    &&R^{uv}(x,x)=-\frac{c_1(x)}{\lambda_1+\lambda_2},\label{r_uu1__r_uvx}\\
    &R^{vv}(\ell,\xi)=\rho R^{uv}(\ell,\xi), 
    &&R^{vu}(x,x)=\frac{c_2(x)}{\lambda_1+\lambda_2}.\label{r_vv1__r_vux}
\end{align}
The well-posedness of the systems \eqref{p_alphaalpha}-\eqref{p_betabeta1__p_betaalphax} and \eqref{r_uu}-\eqref{r_vv1__r_vux} is established in \cite{Vazquez2011} (use the coordinate change $\bar{x}=\ell-\xi,\ \bar{\xi}=\ell-x$).

In order to derive a stabilizing control law via PDE backstepping considering the observer system \eqref{obs_u}-\eqref{obs_BC2}, the following backstepping transformations are introduced:
\begin{align}
\hat{\alpha}(x,t) =&\hat{u}(x,t)-\int_0^xK^{uu}(x,\xi)\hat{u}(\xi,t)d\xi \nonumber\\&-\int_0^xK^{uv}(x,\xi)\hat{v}(\xi,t)d\xi,\label{obs_backstepping_1}\\
\hat{\beta}(x,t)  =&\hat{v}(x,t)-\int_0^x K^{vu}(x, \xi) \hat{u}(\xi,t) d\xi\nonumber\\&-\int_0^x K^{vv}(x, \xi) \hat{v}(\xi,t) d\xi,\label{obs_backstepping_2} 
\end{align}
which are defined over the triangular domain $0 \leq \xi \leq x \leq \ell$, with the kernels $K^{uu},K^{uv},K^{vu},$ and $K^{vv}$ satisfying the following PDEs:
\begin{align}
\lambda_1\partial_xK^{uu}(x,\xi)+\lambda_1\partial_\xi K^{uu}(x,\xi)=&-c_2(\xi)K^{uv}(x,\xi)\label{k_uu},\\
\lambda_1\partial_xK^{uv}(x,\xi)-\lambda_2\partial_\xi K^{uv}(x,\xi)=&-c_1(\xi)K^{uu}(x,\xi)\label{k_uv},\\
\lambda_2\partial_xK^{vu}(x,\xi)-\lambda_1\partial_\xi K^{vu}(x,\xi)=&c_2(\xi)K^{vv}(x,\xi)\label{k_vu},\\
\lambda_2\partial_xK^{vv}(x,\xi)+\lambda_2\partial_\xi K^{vv}(x,\xi)=&c_1(\xi)K^{vu}(x,\xi)\label{k_vv},
\end{align}
and boundary conditions
\begin{align}
&K^{uu}(x,0)=\frac{\lambda_2}{q\lambda_1}K^{uv}(x,0),
&&K^{uv}(x,x)=\frac{c_1(x)}{\lambda_1+\lambda_2},\label{k_uu0__k_uvx}\\
&K^{vv}(x,0)=\frac{q\lambda_1}{\lambda_2}K^{vu}(x,0),
&&K^{vu}(x,x)=-\frac{c_2(x)}{\lambda_1+\lambda_2}.\label{k_vv0__k_vux}
\end{align}
Let us choose the control input $U(t)$ as
\begin{align}
    U(t)=\int_0^\ell N^u(\xi) \hat{u}(\xi,t) d\xi + \int_0^\ell N^v(\xi) \hat{v}(\xi,t) d\xi,\label{cntns_fb_uv}
\end{align}
where
\begin{align}
    N^u(\xi) =& K^{vu}(\ell,\xi) - \rho K^{uu} (\ell,\xi),\\
    N^v(\xi) =& K^{vv}(\ell,\xi) - \rho K^{uv} (\ell,\xi).
\end{align}
Then, the observer \eqref{obs_u}-\eqref{obs_BC2} gets transformed into the following target system:
\begin{align}
\partial_t \hat{\alpha}(x,t) =&-\lambda_1 \partial_x \hat{\alpha}(x,t) + \bar{p}_1(x)\tilde{\beta}(0,t),\label{obs_target_u} \\
\partial_t \hat{\beta}(x,t)=&\lambda_2 \partial_x \hat{\beta}(x,t) + \bar{p}_2(x)\tilde{\beta}(0,t),\label{obs_target_v}    
\end{align}
with boundary conditions
\begin{align}
\hat{\alpha}(0,t)=&q\hat{\beta}(0,t) + q\tilde{\beta}(0,t),\label{obs_target_BC1}\\
\hat{\beta}(\ell,t)=&\rho\hat{\alpha}(\ell,t), \label{obs_target_BC2}
\end{align}
where
\begin{align}
    \bar{p}_1(x)=&p_1(x)-\int_0^xK^{uu}(x,\xi)p_1(\xi)d\xi\nonumber\\&-\int_0^xK^{uv}(x,\xi)p_2(\xi)d\xi,\label{p1_bar}\\
    \bar{p}_2(x)=&p_2(x)-\int_0^xK^{vu}(x,\xi)p_1(\xi)d\xi\nonumber\\&-\int_0^xK^{vv}(x,\xi)p_2(\xi)d\xi.\label{p2_bar}
\end{align}


The inverse transformations of \eqref{obs_backstepping_1},\eqref{obs_backstepping_2} are given by
\begin{align}
    \hat{u}(x,t)  =&\hat{\alpha}(x,t)+\int_0^x L^{\alpha\alpha}(x, \xi) \hat{\alpha}(\xi,t) d\xi \nonumber\\&+ \int_0^x L^{\alpha\beta} \hat{\beta}(\xi,t) d\xi, \label{backstepping_1-inverse}\\
    \hat{v}(x,t)  =&\hat{\beta}(x,t)+\int_0^x L^{\beta\alpha}(x, \xi) \hat{\alpha}(\xi,t) d\xi \nonumber\\&+ \int_0^x L^{\beta\beta}(x,\xi) \hat{\beta}(\xi,t) d\xi, \label{backstepping_2-inverse}
\end{align}
defined over the triangular domain $0 \leq \xi \leq x \leq \ell$, where the kernels $L^{\alpha\alpha},L^{\alpha\beta},L^{\beta\alpha},L^{\beta\beta}$ satisfy the following PDEs:
\begin{align}
    \lambda_1 \partial_x L^{\alpha\alpha}(x,\xi) + \lambda_1 \partial_\xi L^{\alpha\alpha}(x,\xi)=&c_1(x)L^{\beta\alpha}(x,\xi),\label{l_alphaalpha}\\
    \lambda_1 \partial_x L^{\alpha\beta}(x,\xi) - \lambda_2 \partial_\xi L^{\alpha\beta}(x,\xi)=&c_1(x)L^{\beta\beta}(x,\xi),\label{l_alphabeta}\\
    \lambda_2 \partial_x L^{\beta\alpha}(x,\xi) - \lambda_1 \partial_\xi L^{\beta\alpha}(x,\xi)=&-c_2(x)L^{\alpha\alpha}(x,\xi),\label{l_betaalpha}\\
    \lambda_2 \partial_x L^{\beta\beta}(x,\xi) + \lambda_2 \partial_\xi L^{\beta\beta}(x,\xi)=&-c_2(x)L^{\alpha\beta}(x,\xi),\label{l_betabeta}
\end{align}
with boundary conditions
\begin{align}
    &L^{\alpha\alpha}(x,0)=\frac{\lambda_2}{q\lambda_1}L^{\alpha\beta}(x,0),
    &&L^{\alpha\beta}(x,x)=\frac{c_1(x)}{\lambda_1+\lambda_2},\label{l_alphaalpha0__l_alphabetax}\\
    &L^{\beta\beta}(x,0)=\frac{q\lambda_1}{\lambda_2}L^{\beta\alpha}(x,0),
    &&L^{\beta\alpha}(x,x)=-\frac{c_2(x)}{\lambda_1+\lambda_2}.\label{l_betabeta0__l_betaalphax}
\end{align}
The well-posedness of the systems \eqref{k_uu}-\eqref{k_vv0__k_vux} and \eqref{l_alphaalpha}-\eqref{l_betabeta0__l_betaalphax} is established in \cite{Vazquez2011}. 

It is worth noting that the control input $U(t)$ given by \eqref{cntns_fb_uv} can also be expressed as
\begin{equation}\label{cntns_fb}
    U(t)=\int_0^\ell N^\alpha(\xi) \hat{\alpha}(\xi,t)d\xi+\int_0^\ell N^\beta(\xi) \hat{\beta}(\xi,t)d\xi,
\end{equation}
where
\begin{align}
    N^\alpha(\xi)=&L^{\beta\alpha}(\ell,\xi)-\rho L^{\alpha\alpha}(\ell,\xi),\label{n_alpha}\\
    N^\beta(\xi)=&L^{\beta\beta}(\ell,\xi)-\rho L^{\alpha\beta}(\ell,\xi).\label{n_beta}
\end{align}

\begin{propo}
    Subject to Assumption \ref{assum:reflection}, the continuous-time plant and the observer \eqref{sys_u}-\eqref{sys_BC2},\eqref{obs_u}-\eqref{obs_BC2} with the continuous-time  control input \eqref{cntns_fb_uv} has a unique solution $(u,v,\hat 
    {u}, \hat{v})^T \in \mathcal{C}^0([t_k^\omega,t_{k+1}^\omega]; L^2((0,\ell);\mathbb{R}^4))$ and is globally exponentially stable in the spatial $L^2$ norm. 
\end{propo}
\begin{proof}
    Since the closed-loop system \eqref{sys_u}-\eqref{sys_BC2},\eqref{obs_u}-\eqref{obs_BC2} is well-known to be well-posed under a continuous boundary input, there exists a unique solution $(u,v,\hat 
    {u}, \hat{v})^T \in \mathcal{C}^0([t_k^\omega,t_{k+1}^\omega]; L^2((0,\ell);\mathbb{R}^4))$. To establish the stability of system \eqref{sys_u}-\eqref{sys_BC2},\eqref{obs_u}-\eqref{obs_BC2},\eqref{cntns_fb_uv}, we first prove the $L^2$ exponential stability of the target systems \eqref{err_target_u}-\eqref{err_target_BC2}, \eqref{obs_target_u}-\eqref{obs_target_BC2}, which induce the $L^2$ exponential stability of the original closed-loop system. Consider the following Lyapunov function:
\begin{align}
    W_0(t)=V_0^1(t)+C_0V_0^2(t),\label{w_0}
\end{align}
where
\begin{align*}
    V_0^1(t)=&\int_0^\ell \left(\frac{A_0}{\lambda_1}\tilde{\alpha}^2(x,t)e^{-\mu_0\frac{x}{\lambda_1}}
    + \frac{B_0}{\lambda_2}\tilde{\beta}^2(x,t)e^{\mu_0\frac{x}{\lambda_2}}\right)dx,\\
    V_0^2(t)=&\int_0^\ell \left(\frac{A_0}{\lambda_1}\hat{\alpha}^2(x,t)e^{-\mu_0\frac{x}{\lambda_1}}
    + \frac{B_0}{\lambda_2}\hat{\beta}^2(x,t)e^{\mu_0\frac{x}{\lambda_2}}\right)dx.\\
\end{align*}
Let us select the constants $\mu_0$, $A_0$, $B_0$, and $C_0$ as
\begin{align*}
    \mu_0\in&\left(0,\frac{2\lambda_1\lambda_2}{\ell(\lambda_1+\lambda_2)}\ln{\left(\frac{1}{\sqrt{2}|\rho q|}\right)}\right),\\
    A_0=&\rho^2e^{\mu_0\left(\frac{\ell}{\lambda_1}+\frac{\ell}{\lambda_2}\right)}+\frac{e^{\mu_0\frac{\ell}{\lambda_1}}}{2},\\
    B_0=&q^2e^{\mu_0\frac{\ell}{\lambda_1}}+1,\\
    C_0=&\left[\frac{2r_0}{\mu_0}\int_0^\ell\left(\frac{A_0}{\lambda_1}e^{-\mu_0\frac{x}{\lambda_1}}\bar{p}_1^2(x) + \frac{B_0}{\lambda_2}e^{\mu_0\frac{x}{\lambda_2}}\bar{p}_2^2(x)\right)\right]^{-1},\\
    \delta_0=&\frac{2r_0}{\mu_0}\left(1-2\rho^2q^2e^{\mu_0\left(\frac{\ell}{\lambda_1}+\frac{\ell}{\lambda_2}\right)}\right),
\end{align*}
where $r_0$ is selected large enough such that $C_0<1$ and $1-\frac{2}{\mu_0\delta_0}>0$. Then, under Assumption~\ref{assum:reflection}, using \eqref{err_target_u}-\eqref{err_target_BC2}, \eqref{obs_target_u}-\eqref{obs_target_BC2}, and using Young's inequality, we obtain the following relation for $\dot{W}_0(t)$:
\begin{align*}
    \dot{W}_0(t)\leq-\frac{\mu_0}{2}W_0(t).
\end{align*}
 Hence, we see that the target systems \eqref{err_target_u}-\eqref{err_target_BC2} and \eqref{obs_target_u}-\eqref{obs_target_BC2} are globally exponentially stable in the spatial $L^2$ norm. Using the boundedness of the backstepping transformations \eqref{err_backstepping1},\eqref{err_backstepping2}, \eqref{err_backstepping1-inverse},\eqref{err_backstepping2-inverse},\eqref{obs_backstepping_1},\eqref{obs_backstepping_2},\eqref{backstepping_1-inverse}, and \eqref{backstepping_2-inverse}, we can conclude that the closed-loop system \eqref{sys_u}-\eqref{sys_BC2},\eqref{obs_u}-\eqref{obs_BC2} along with the continuous-time control input $U(t)$ given by \eqref{cntns_fb_uv}, is globally exponentially stable in the spatial $L^2$ norm. \hfill 
\end{proof}

\begin{figure}[t]
    \centering
    \begin{tikzpicture}

\def\len{0.6}
\def\height{1.8}
\def\numconnections{5}

\pgfmathsetmacro{\width}{\len*\columnwidth}
\pgfmathsetmacro{\connectdist}{\width/(\numconnections+1)}
\pgfmathsetmacro{\vertdist}{\height*\baselineskip}

\node (u0) [label={[xshift=10 pt, yshift=-1 pt]$u(x,t)$}]{};
\node (u1) [right=\width pt of u0]{};
\node (v0) [below = \vertdist pt of u0, label={[xshift=10 pt, yshift=-1 pt]$v(x,t)$}]{};
\node (v1) [right=\width pt of v0]{};

\foreach \i in {1,...,\numconnections} {
	\node (nodetop\i) [right=\i*\connectdist pt of u0]{};
	\node (nodebottom\i) [right=\i*\connectdist pt of v0]{};
}		

\node (x0) [below =5pt of v0, label={[xshift=0 pt, yshift=-20 pt]$x=0$}] {};
\node (x1) [below =5pt of v1, label={[xshift=0 pt, yshift=-20 pt]$x=\ell$}] {};
\node (x1a) [right =0.5 pt of x1] {};

\node (y0) [left = 2 pt of v0] {};
\node (y1) [left = 10 pt of y0] {};

\node (Uk0) [right =2 pt of v1] {};
\node (Uk1) [right = 5 pt of Uk0] {};
\node (Uk2) [below =1.5* \vertdist pt of Uk1] {};

\node (c0) [below =1.5* \vertdist pt of y1] {};
\node (c1) [right = 0.5*\width pt of c0] {};
\node (e0) [below= 1.5*\vertdist pt of c0] {};
\node (e1) [right = 0.5*\width pt of e0] {};

\node (obs) [rectangle, draw = blue, text = blue, minimum width = 2cm, minimum height = 0.6cm] at (c1) {Observer + Controller};
\node (trig) [rectangle, draw = red, text = red, minimum width = 2cm, minimum height = 0.6cm] at (e1) {Triggering mechanism};

\node(c2)[right = 0.1*\width pt of obs.east]{};

	\node(sw) [below right = 0.7071*0.05*\width pt and  0.7071*0.05*\width pt of c2] {};

\node(e2)[right = 0.125*\width pt of trig.east]{};
\node(e3)[above = 1.4*\vertdist pt of e2.center]{};

\node(Uk3) [right=0.05*\width pt of c2]{};
\draw [-stealth, line width=0.8pt] (x0.west) to (x1a.east);
\draw [line width=0.5pt] (x0.north) -- (x0.south);
\draw [line width=0.5pt] (x1.north) -- (x1.south);
\foreach \i in {1,...,\numconnections} {
	\draw [stealth-stealth,dashed, line width=1pt] (nodetop\i.south) to (nodebottom\i.north);
}

\draw [-stealth, line width=2pt] (u0.center) to (u1.center);
\draw [stealth-, line width=2pt] (v0.center) to (v1.center);

\draw [-stealth, line width=1 pt, bend left=45] (u1.east) to node[pos=0.5, right]{$\rho$} (v1.east);
\draw [-stealth, line width=1 pt, bend left=45] (v0.west) to node[pos=0.5, left]{$q$} (u0.west);

\draw [blue,-,line width=1pt] (y0.center) -- node[pos=1.2, above]{$v(0,t)$} (y1.center);
\draw [blue,-,line width=1pt] (y1.center) --  (c0.center);
\draw [blue,-stealth,line width=1pt] (c0.center) -- (obs.west);

\draw [blue, line width=1pt,decoration={markings,mark=at position 0.6 with {\arrow{stealth}}},  postaction={decorate}] (obs.east)  -- (c2.center); 
\draw[red,-stealth,line width=1pt] ([xshift=10pt, yshift=0pt]obs.south) -- node[pos=0.5, left]{$\hat{\alpha}(x,t),\hat{\beta}(x,t)$}([xshift=10pt, yshift=0pt]trig.north);
\draw [red,-stealth, line width=1pt] (c0.center)  |- (trig.west); 

\draw [blue,-,line width=2pt] (c2.center) -- (sw.center);

\draw [fill=blue,draw=blue](c2.center) circle [radius=1pt]; 

\draw [red,-,dashed, line width=1pt](trig.east) -- (e2.center);
 \draw [red,-stealth,dashed, line width=1pt] (e2.center)  -- node[pos=0.2, right]{$\{t_k^\omega\}_{k\in\mathbb{N}}$} (e3.center);

\draw [red,-, line width=1pt] (Uk3.center) -- node[pos=0.5, below]{$U_k^\omega(t)$} (Uk2.center);
\draw [red,-, line width=1pt] (Uk2.center) -- (Uk1.center);
\draw [red,-stealth, line width=1pt] (Uk1.center) -- (Uk0.center);

\draw [fill=red,draw=red](Uk3.center) circle [radius=1pt]; 
\end{tikzpicture}
    \caption{Observer-based event-triggered closed-loop system. The increasing sequences of event times $\{t_k^c\}_{k\in\mathbb{N}}$, $\{t_k^p\}_{k\in\mathbb{N}}$, and $\{t_k^p\}_{k\in\mathbb{N}}$ with $\omega=\{``c",``p",``s"\}$ are generated by CETC, PETC, and STC triggering mechanisms.}
    \label{fig:etc}
\end{figure}
 From \eqref{cntns_fb_uv}, we derive by emulation the following aperiodic sampled-data control signal held constant between events:
\begin{align}
    U_k^\omega(t):=U(t_k^\omega),\label{U_d}
\end{align}
for all $t\in[t_k^\omega,t_{k+1}^\omega), k\in\mathbb{N}$, $\omega=\{``c",``p",``s"\}$. Entities related to CETC, PETC, and STC are labeled by the superscripts $``c"$, $``p"$, and $``s"$ respectively. In subsequent sections, we present event-triggering rules to determine increasing sequences of event times, $I^\omega=\{t_k^\omega\}_{k\in\mathbb{N}}$. The event-triggered closed-loop system is depicted in Fig.~\ref{fig:etc}. The deviation between the sampled-data control \eqref{U_d} and the continuous control \eqref{cntns_fb_uv}, referred to as the input holding error, is defined as
\begin{align}
    d(t):=&U_k^\omega(t)-U(t),\label{d}
\end{align}
for $t\in[t_k^\omega,t_{k+1}^\omega), k\in\mathbb{N}$.
Under the control input $U_k^\omega(t)$, the boundary values \eqref{sys_BC2}, \eqref{obs_BC2}, and \eqref{obs_target_BC2} change to
\begin{align}
    v(\ell,t)=&\rho u(\ell,t) + U_k^\omega(t),\label{sys_BC2_d}\\
    \hat{v}(\ell,t)=&\rho \hat{u}(\ell,t) + U_k^\omega(t),\label{obs_BC2_d}\\
    \hat{\beta}(\ell,t)=&\rho\hat{\alpha}(\ell,t) + d(t).\label{obs_target_BC2_d}
\end{align}
In the following proposition, we establish the well-posedness of the systems \eqref{sys_u}-\eqref{sys_BC1},\eqref{sys_BC2_d} and \eqref{obs_u}-\eqref{obs_BC1},\eqref{obs_BC2_d} between consecutive events.
\begin{propo}\label{prop:well-posed}
    Let $k\in\mathbb{N}$, and $U_{k}^{\omega}(t)\in\mathbb{R}$ be constant between two event times $t_k^\omega$ and $t_{k+1}^\omega$. For a given $(u(\cdot,t_k^\omega), v(\cdot,t_k^\omega))^T \in L^2((0,\ell);\mathbb{R}^2)$ and $(\hat u(\cdot,t_k^\omega), \hat v(\cdot,t_k^\omega))^T \in L^2((0,\ell);\mathbb{R}^2)$, there exist unique solutions such that $(u, v)^T \in \mathcal{C}^0([t_k^\omega, t_{k+1}^\omega];  L^2((0,\ell);\mathbb{R}^2))$ and $(\hat 
    {u}, \hat{v})^T \in \mathcal{C}^0([t_k^\omega,t_{k+1}^\omega]; L^2((0,\ell);\mathbb{R}^2))$ to the systems \eqref{sys_u}-\eqref{sys_BC1},\eqref{sys_BC2_d} and \eqref{obs_u}-\eqref{obs_BC1},\eqref{obs_BC2_d} respectively between two time instants $t_k^\omega$ and $t_{k+1}^\omega$.  
\end{propo}
The proof of Proposition~\ref{prop:well-posed} is similar to \cite[Proposition 1]{espitiaObserverbased2020}. 

We now proceed to develop the following result on the existence and uniqueness of solutions of the systems \eqref{sys_u}-\eqref{sys_BC1},\eqref{sys_BC2_d} and \eqref{obs_u}-\eqref{obs_BC1},\eqref{obs_BC2_d} with the input \eqref{U_d} for all $t\in\mathbb{R}^+$.
\begin{corollary}\label{cor:existece}
     Let $k\in\mathbb{N}$, and $U_{k}^{\omega}(t)\in\mathbb{R}$ be constant between two event times $t_k^\omega$ and $t_{k+1}^\omega$ in the sequence of event times $I^\omega=\{t_k^\omega\}_{k\in\mathbb{N}}$, and assume that for all such intervals, there exists a lower bound $\tau_0>0$ such that $t_{k+1}^\omega-t_k^\omega\geq\tau_0$. Then, for every initial condition $(u^0,v^0)^T\in L^2((0,\ell);\mathbb{R}^2)$ and $(\hat{u}^0,\hat{v}^0)^T\in L^2((0,\ell);\mathbb{R}^2)$, there exist unique solutions $(u,v)^T\in\mathcal{C}^0(\mathbb{R}^+;L^2((0,\ell);\mathbb{R}^2))$ and $(\hat{u},\hat{v})^T\in\mathcal{C}^0(\mathbb{R}^+;L^2((0,\ell);\mathbb{R}^2))$ to the systems \eqref{sys_u}-\eqref{sys_BC1},\eqref{sys_BC2_d} and \eqref{obs_u}-\eqref{obs_BC1},\eqref{obs_BC2_d} respectively.
\end{corollary}
\begin{proof}
    Note that this proof follows the proof of \cite[Proposition 1]{espitiaObserverbased2020}. Since we assume that for $k\in\mathbb{N}$, $t_{k+1}^\omega-t_k^\omega\geq\tau_0>0$, for all events in $I^\omega\{t_k^\omega\}_{k\in\mathbb{N}}$ we know that as $k\rightarrow\infty$, $t_k^\omega\rightarrow\infty$. Hence a solution can be constructed by iteratively applying Proposition~\ref{prop:well-posed} for all intervals in $I^\omega$ for $t\in\mathbb{R}^+$. \hfill 
\end{proof}

Considering an interval $t\in(t_k^\omega,t_{k+1}^\omega), k\in\mathbb{N}$ and using Proposition~\ref{prop:well-posed},  the following lemma holds for $d(t)$ given by \eqref{d}.
\begin{lemm}\label{lem:ddot2}
    For $d(t)$ given by \eqref{d}, the following inequality holds for all $t\in(t_k^\omega,t_{k+1}^\omega), k\in\mathbb{N}$:
    \begin{align}
         \nonumber
        (\dot{d}(t))^2\leq&\epsilon_0\|(\hat{\alpha}(\cdot,t),\hat{\beta}(\cdot,t))^T\|^2 + \epsilon_1 \hat{\alpha}^2(\ell,t)\nonumber\\ &+ \epsilon_2\tilde{\beta}^2(0,t) + \epsilon_3 d^2(t),\label{ddot2}
    \end{align}
     where $\epsilon_0,\epsilon_1,\epsilon_2,$ and $\epsilon_3\geq0$ are given by
    \begin{align}
        \epsilon_0=&5\max\bigg\{\lambda_1^2\int_0^\ell(\partial_\xi N^\alpha(\xi))^2d\xi,\nonumber\\
&\mkern150mu\lambda_2^2\int_0^\ell(\partial_\xi N^\beta(\xi))^2d\xi\bigg\},\label{epsilon0}\\
\epsilon_1=&5(\lambda_1N^\alpha(\ell)-\rho\lambda_2N^\beta(\ell))^2,\,\label{epsilon1}\\
\epsilon_2=&5\Bigg(\int_0^\ell\left(N^\alpha(\xi)\bar{p}_1(\xi)+N^\beta(\xi)\bar{p}_2(\xi)\right)d\xi\nonumber\\
        & \mkern150mu+q\lambda_1N^\alpha(0)\Bigg)^2,\label{epsilon2}\\
        \epsilon_3=&5(\lambda_2N^\beta(\ell))^2.\label{epsilon3}
    \end{align}
\end{lemm}
\begin{proof}
    Note that this proof follows the proof of \cite[Lemma 2]{espitiaObserverbased2020}. Taking the time derivative of $d(t)$ for $t\in(t_k^\omega,t_{k+1}^\omega), k\in\mathbb{N}$, and using \eqref{obs_target_u},\eqref{obs_target_v}, we obtain the following relation:
    \begin{align}
        \dot{d}(t)=&\lambda_1\int_0^\ell N^\alpha(\xi)\partial_\xi\hat{\alpha}(\xi,t)d\xi-\lambda_2\int_0^\ell N^\beta(\xi)\partial_\xi\hat{\beta}(\xi,t)d\xi\nonumber\\
        &-\tilde{\beta}(0,t)\int_0^\ell \left(N^\alpha(\xi)\bar{p}_1(\xi)d\xi + N^\beta(\xi)\bar{p}_2(\xi)\right)d\xi.
    \end{align}
    Integrating by parts with the help of  \eqref{obs_target_BC1},\eqref{obs_target_BC2}, using Young's  inequality,  and the fact that $\lambda_2N^\beta(0)-q\lambda_1N^\alpha(0)=0$, one gets
  
    \begin{align}\label{d-estim}
        &(\dot{d}(t))^2\leq 5\bigg\{(\lambda_1N^\alpha(\ell)-\rho\lambda_2N^\beta(\ell))^2\hat{\alpha}^2(\ell,t)\nonumber\\
        &+(\lambda_2N^\beta(\ell))^2d^2(t)
        +\lambda_1^2\left(\int_0^\ell\partial_\xi N^\alpha(\xi)\hat{\alpha}(\xi,t)d\xi\right)^2\nonumber\\
        &+ \lambda_2^2\left(\int_0^\ell\partial_\xi N^\beta(\xi)\hat{\beta}(\xi,t)d\xi\right)^2\nonumber\\
        &+\tilde{\beta}(0,t)^2 \bigg(\int_0^\ell N^\alpha(\xi)\bar{p}_1(\xi)d\xi+\int_0^\ell N^\beta(\xi)\bar{p}_2(\xi)d\xi\nonumber\\
        &+q\lambda_1N^\alpha(0)\bigg)^2\bigg\}.
    \end{align}
    Hence, using the Cauchy–Schwarz inequality, we can obtain from \eqref{d-estim} the estimate \eqref{ddot2}, where $\epsilon_0,\epsilon_1,\epsilon_2,\epsilon_3\geq0$ are given by \eqref{epsilon0}-\eqref{epsilon3}. This concludes the proof. \hfill 
\end{proof} 
\subsection{Continuous-time  event-triggered control (CETC)}\label{subsec:cetc}
This section details the CETC triggering mechanism that determines the increasing sequence of event times $I^c=\{t_k^c\}_{k\in\mathbb{N}}$ at which the control input $U_k^c(t)$ is updated by continuously evaluating a triggering function $\Gamma^c(t)$. The sequence  $I^c$ is determined via the following rule.
\begin{definition}\label{defn:triggering}
The increasing sequence of event times $I^c=\{t_k^c\}_{k\in\mathbb{N}}$ with $t_0^c=0$ are determined via the following rule:
\begin{align}
t_{k+1}^c:=&\inf\{t\in\mathbb{R}|t>t_k^c, \Gamma^c(t)>0,k\in\mathbb{N}\},\label{tck}\\
\Gamma^c(t):=&\theta d^2(t) + m(t),\label{gamma_c}
\end{align}
where $d(t)$ is given by \eqref{d}. The dynamic variable $m(t)$ evolves according to the ODE
\begin{align}
    \dot{m}(t)=&-\eta m(t) +\theta_m d^2(t)- \kappa_0 \|(\hat{\alpha}(\cdot,t),\hat{\beta}(\cdot,t))^T\|^2 \nonumber\\& -\kappa_1\hat{\alpha}^2(\ell,t)-\kappa_2\tilde{\beta}^2(0,t),\label{m}
\end{align}
for $t\in(t_k^c,t_{k+1}^c),k\in\mathbb{N}$ with $m(0)=m^0<0,~m(t_k^{c^-})=m(t_k^c)=m(t_k^{c^+})~k\in\mathbb{N}$. Let $\eta,\theta>0$ be arbitrary parameters and $\kappa_0,\kappa_1,\kappa_2>0$, and $\theta_m>0$ be event-trigger parameters to be determined.
\end{definition}

It should be noted that the CETC design follows the design procedure in \cite{espitiaObserverbased2020} with the only difference being that the control input is anti-collocated with the measurement.

The event-triggering rule guarantees that $\Gamma^c(t)\leq 0$ for all $t\in[0,F)$, where $F=\sup\{I^c\}$, hence the following lemma holds.
\begin{lemm}\label{lem:m<0} Under the CETC approach \eqref{U_d},\eqref{tck}-\eqref{m}, the dynamic variable $m(t)$ with $m^0<0$ satisfies $m(t)<0$ for all $t\in[0,F)$, where $F=\sup\{I^c\}$. 
\end{lemm}
\begin{proof}
    Note that this proof follows the proof of \cite[Lemma 1]{espitiaObserverbased2020}. The event-triggering rule ensures that $\theta d^2(t) \leq -m(t)$. Thus, using \eqref{m}, we can obtain the following inequality :
    \begin{align}
        \dot{m}(t)\leq&-\left(\eta+\frac{\theta_m}{\theta}\right) m(t) - \kappa_0 \|(\hat{\alpha}(\cdot,t),\hat{\beta}(\cdot,t))^T\|^2 \nonumber\\& -\kappa_1\hat{\alpha}^2(\ell,t)-\kappa_2\tilde{\beta}^2(0,t),
    \end{align}
    for all $t\in(t_k^c,t_{k+1}^c),k\in\mathbb{N}$. Assume that $m(t_k^c)<0$. Considering the continuity of $m(t)$ and using the comparison principle, it can be seen that $m(t)<0$ for $t\in(t_k^c,t_{k+1}^c),k\in\mathbb{N}$. Since $m(t)$ is defined such that $m(t_k^{c^-})=m(t_k^c)=m(t_{k+1}^{c^+})$, it holds that  $m(t)<0$ for all $t\in[t_k^c,t_{k+1}^c],k\in\mathbb{N}$. Successively using the same argument for all intervals in $I^c$ and recalling that $m^0<0$, we obtain $m(t)<0$ for all $t\in[0,F)$, where $F=\sup\{I^c\}$. \hfill 
\end{proof}

The existence of a minimum dwell-time $\tau>0$ is shown in the following lemma.

\begin{lemm}\label{lem:zeno} Let the CETC events be triggered according to the rule \eqref{tck}-\eqref{m}. 
Furthermore, let $\sigma\in(0,1)$ and $\theta$ be free parameters, and $\kappa_0,\kappa_1,\kappa_2>0$ satisfy
    \begin{align}
        \theta \epsilon_i=(1-\sigma)\kappa_i,\quad\text{for}~i=0,1,2,\label{eqkappa_i}
    \end{align}
where $\epsilon_0,\epsilon_1,\epsilon_2$ are given by \eqref{epsilon0}-\eqref{epsilon2}, respectively. Then, there exists a minimum dwell-time $\tau>0$ such that $t_{k+1}^c-t_k^c\geq\tau$ for all $k\in\mathbb{N}$.
\end{lemm}
\begin{proof}
    Note that this proof follows the proof of \cite[Theorem 1]{espitiaObserverbased2020}. Due to the triggering rule \eqref{tck},\eqref{gamma_c}, the following holds:
    \begin{align}
        \theta d^2(t) \leq -\sigma m(t) -(1-\sigma)m(t).
    \end{align}
    Rearranging this inequality, let us define
    \begin{align}
        \psi(t):=\frac{\theta d^2(t) + (1-\sigma)m(t)}{-\sigma m(t)}.\label{psi}
    \end{align}
    Let an event be triggered at time $t=t_k^c,~k\in\mathbb{N}$, then, $d(t_k^c)=0$ and since $\sigma\in(0,1)$ and $m(t_k^c)<0$, it can be seen that $\psi(t_k^{c})<0$ and $\psi(t_{k+1}^{c^-})=1$. Also note that $d(t)$ and $m(t)$ are continuous on $t\in(t_k^c,t_{k+1}^c)$ making $\psi(t)$ a continuous function on $t\in(t_k^c,t_{k+1}^{c})$. Therefore, from the intermediate value theorem, $\exists t_k^\prime\in(t_k^c,t_{k+1}^{c})$ such that $\psi(t)\in[0,1]$ for $t\in[t_k^\prime,t_{k+1}^{c^-}]$. Taking the time derivative of $\psi(t)$ on $t\in(t_k^\prime,t_{k+1}^{c^-})$ and using \eqref{eqkappa_i}, we can derive the following expression following a similar procedure as in \cite[Theorem 1]{espitiaObserverbased2020}:
    \begin{align}
        \dot{\psi}(t)\leq a_0 + a_1 \psi(t) + a_2 \psi^2(t),
    \end{align}
    where  $a_0,a_1,a_2>0$ are defined as 
\begin{align}
    a_0=&\left(1+\epsilon_3+\eta +\frac{\theta_m(1-\sigma)}{\theta}\right)\frac{(1-\sigma)}{\sigma},\\
    a_1=&1+\epsilon_3+\eta+\frac{2\theta_m(1-\sigma)}{\theta},\quad 
    a_2=\frac{\theta_m\sigma}{\theta}.
\end{align}
Therefore, using the comparison principle, the time required for $\psi(t)$ to change from $\psi(t_k^\prime)=0$ to $\psi(t_{k+1}^{c^-})=1$ is at least
    \begin{align}
        \tau=\int_0^1\frac{1}{a_0+a_1s+a_2s^2}ds.\label{tau_int}
    \end{align}
    Performing partial fraction expansion  and computing  \eqref{tau_int}, one gets
    \begin{align}
        \tau=\frac{1}{a}\ln\left(1+\frac{a\theta\sigma}{(a\theta + \theta_m )(1-\sigma)}\right),\label{tau}
    \end{align}
   where 
   \begin{equation}\label{aaaa}
       a=1+\epsilon_3+\eta>0
   \end{equation}  
Since $t_{k+1}^c-t_k^\prime>\tau$ and $t_{k+1}^c-t_k^c>t_{k+1}^c-t_k^\prime$, $t_{k+1}^c-t_k^c>\tau$, we can consider $\tau>$ as the minimum dwell-time, namely, the lower bound for the time between two triggering events, which excludes the occurrence  of the Zeno phenomenon for the event-triggering rule \eqref{tck}-\eqref{m}. \hfill 
\end{proof}

We establish the global exponential convergence of closed-loop system \eqref{sys_u}-\eqref{sys_BC1},\eqref{obs_u}-\eqref{obs_BC1},\eqref{U_d},\eqref{sys_BC2_d},\eqref{obs_BC2_d} in the following theorem.
\begin{propo}\label{propo:lyapunov stability}
Subject to Assumption \ref{assum:reflection}, let $\eta,\theta>0$ and the parameter $\theta_m$, be selected such that
\begin{align}
    \theta_m=2De^{\mu\frac{\ell}{\lambda_2}}, \label{thetam}
\end{align}
where
\begin{align}
    D=&2Cq^2,\label{D}\\
    C>&\max\left\{\frac{\kappa_0}{(\mu-\delta)r},\frac{\kappa_1}{1-4\rho^2q^2e^{\mu\left(\frac{\ell}{\lambda_1}+\frac{\ell}{\lambda_2}\right)}}\right\},\label{C}\\
    \mu\in&\left(0,\frac{2\lambda_1\lambda_2}{\ell(\lambda_1+\lambda_2)}\ln\left(\frac{1}{2|q\rho|}\right)\right),\quad \delta<\mu,\label{mu,delta}\\
    r=&\min\left\{\frac{1}{\lambda_1}e^{-\mu\frac{\ell}{\lambda_1}},\frac{2q^2}{\lambda_2}\right\},\label{r}
\end{align}
and $\kappa_0$, $\kappa_1>0$ satisfy \eqref{eqkappa_i}. Then, under the CETC triggering rule \eqref{tck}-\eqref{m}, the observer-based CETC closed-loop system \eqref{sys_u}-\eqref{sys_BC1},\eqref{obs_u}-\eqref{obs_BC1},\eqref{U_d},\eqref{sys_BC2_d},\eqref{obs_BC2_d} has a unique solution $(u,v,\hat{u}, \hat{v})^T \in \mathcal{C}^0(\mathbb{R}^+; L^2((0,\ell);\mathbb{R}^4))$, and the closed-loop system states globally exponentially converge to $0$ in the spatial $L^2$ norm. 
\end{propo}
\begin{proof}
Using Corollary~\ref{cor:existece} and Lemma~\ref{lem:zeno}, we can conclude that the closed-loop system \eqref{sys_u}-\eqref{sys_BC1},\eqref{obs_u}-\eqref{obs_BC1},\eqref{U_d},\eqref{sys_BC2_d},\eqref{obs_BC2_d}, has a unique solution $(u,v,\hat{u}, \hat{v})^T \in \mathcal{C}^0(\mathbb{R}^+; L^2((0,\ell);\mathbb{R}^4))$. Consider the following Lyapunov function for the systems \eqref{err_target_u}-\eqref{err_target_BC2} and \eqref{obs_target_u}-\eqref{obs_target_BC1},\eqref{obs_target_BC2_d}:
\begin{align}
    W(t)=V_1(t)+V_2(t)-m(t),\label{w}
\end{align}
where
\begin{align}
    V_1(t)=&\int_0^\ell \left(\frac{A}{\lambda_1}\tilde{\alpha}^2(x,t)e^{-\mu\frac{x}{\lambda_1}}
    + \frac{B}{\lambda_2}\tilde{\beta}^2(x,t)e^{\mu\frac{x}{\lambda_2}}\right)dx,\label{v1}\\
    V_2(t)=&\int_0^\ell \left(\frac{C}{\lambda_1}\hat{\alpha}^2(x,t)e^{-\mu\frac{x}{\lambda_1}}
    + \frac{D}{\lambda_2}\hat{\beta}^2(x,t)e^{\mu\frac{x}{\lambda_2}}\right)dx,\label{v2}
\end{align}
with $\mu,A,B,C,D>0$.Note that there exists $\tilde{r}_{v1}$, $\tilde{r}_{v2}$, $\hat{r}_{v1}$, and $\hat{r}_{v2}$ such that
\begin{align}
    &\tilde{r}_{v1}\|\tilde{\alpha}(\cdot,t),\tilde{\beta}(\cdot,t))^T\|^2\leq V_1\leq\tilde{r}_{v2}\|\tilde{\alpha}(\cdot,t),\tilde{\beta}(\cdot,t))^T\|^2,\label{v1_bounds}\\
    &\hat{r}_{v1}\|\hat{\alpha}(\cdot,t),\hat{\beta}(\cdot,t))^T\|^2\leq V_2\leq\hat{r}_{v2}\|\hat{\alpha}(\cdot,t),\hat{\beta}(\cdot,t))^T\|^2,\label{v2_bounds}\\
    &\tilde{r}_{v1}=\min\left\{\frac{A}{\lambda_1}e^{-\mu\frac{\ell}{\lambda_1}},\frac{B}{\lambda_2}\right\},\mkern9mu\tilde{r}_{v2}=\max\left\{\frac{A}{\lambda_1},\frac{B}{\lambda_2}e^{\mu\frac{\ell}{\lambda_2}}\right\},\nonumber\\
    &\hat{r}_{v1}=\min\left\{\frac{C}{\lambda_1}e^{-\mu\frac{\ell}{\lambda_1}},\frac{D}{\lambda_2}\right\},\mkern9mu\hat{r}_{v2}=\max\left\{\frac{C}{\lambda_1},\frac{D}{\lambda_2}e^{\mu\frac{\ell}{\lambda_2}}\right\}.\nonumber
\end{align}

Differentiating \eqref{w} with respect to time and using the equations \eqref{err_target_u}-\eqref{err_target_BC2}, \eqref{obs_target_u}-\eqref{obs_target_BC1},\eqref{obs_target_BC2_d} along with \eqref{v1},\eqref{v2} and \eqref{m} with $\theta_m$ chosen as in \eqref{thetam}, and then using Young's inequality for any $\delta>0$, we obtain the following estimate:
\begin{align}
   & \dot{W}\leq-\mu V_1(t) +\eta m(t) - \left(\mu-\delta-\frac{\kappa_0}{\hat{r}_{v1}}\right)V_2(t)\nonumber\\
   &- (Ae^{-\mu\frac{\ell}{\lambda_1}}-B\rho^2e^{\mu\frac{\ell}{\lambda_2}})\tilde{\alpha}^2(\ell,t)- (D-2Cq^2)\hat{\beta}^2(0,t)\nonumber\\
   &-(Ce^{-\mu\frac{\ell}{\lambda_1}}-2D\rho^2e^{\mu\frac{\ell}{\lambda_2}}-\kappa_1)\hat{\beta}^2(\ell,t)\nonumber\\
   &-\Bigg(-\frac{1}{\delta}\int_0^\ell \frac{C}{\lambda_1}\bar{P}_1^2(x)e^{-\mu\frac{\ell}{\lambda_1}}+\frac{D}{\lambda_1}\bar{P}_2^2(x)e^{\mu\frac{\ell}{\lambda_2}}dx \nonumber\\
   &\mkern200mu+B-2Cq^2\Bigg)\tilde{\beta}^2(0,t).
\end{align}
Thereafter define $\delta<\mu$ and the constants $C$, $D$ as given in \eqref{C},\eqref{D} respectively and $A$, $B$ as given below
\begin{align*}
    A=&B\rho^2e^{\mu\left(\frac{\ell}{\lambda_1}+\frac{\ell}{\lambda_2}\right)},\\
    B=&\kappa_2+D+\frac{1}{\delta}\int_0^\ell\left(\frac{C}{\lambda_1}e^{-\mu\frac{x}{\lambda_1}}\bar{p}_1 ^2(x)  +  \frac{D}{\lambda_2}e^{\mu\frac{x}{\lambda_2}}\bar{p}_2^2(x)\right)dx,
\end{align*}
where
\begin{align*}
    r=\min\left\{\frac{1}{\lambda_1}e^{-\mu\frac{\ell}{\lambda_1}},\frac{2q^2}{\lambda_2}\right\}.
\end{align*}
for any $\theta>0$ and $\sigma\in(0,1)$ define $\kappa_0$ $\kappa_1$, and $\kappa_2$ by \eqref{eqkappa_i}. Hence, we obtain the following upper bound for $\dot{W}(t)$:
\begin{align}
    \dot{W}&\leq-v^*W(t),\\
    v^*&=\min\left\{\mu-\delta-\frac{\kappa_0}{\hat{r}_{v1}},\eta\right\}.\nonumber
\end{align}
Using the comparison principle, the following estimate holds for $W(t)$:
\begin{align}
    W(t)\leq W(0)e^{-v^*t},
\end{align}
from \eqref{w}, and noting that $m(t)<0$, we see that
\begin{align}
    V_1(t)+V_2(t)\leq W(0)e^{-v^*t}.
\end{align}
from \eqref{v1_bounds}, \eqref{v2_bounds}, the following inequality holds for the $L^2$-norm of the states of the system \eqref{err_target_u}-\eqref{err_target_BC2}, \eqref{obs_target_u}-\eqref{obs_target_BC1},\eqref{obs_target_BC2_d}:
\begin{align}
    &\|(\tilde{\alpha}(\cdot,t),\tilde{\beta}(\cdot,t))^T\|^2 + \|(\hat{\alpha}(\cdot,t),\hat{\beta}(\cdot,t))^T\|^2 \leq  \nonumber\\
    &\mkern50mu-\frac{1}{\min\{\tilde{r}_{v1},\hat{r}_{v1}\}}m^0e^{-v^*t}\nonumber\\
    &\mkern50mu+ \frac{\max\{\tilde{r}_{v2},\hat{r}_{v2}\}}{\min\{\tilde{r}_{v1},\hat{r}_{v1}\}} (\|(\tilde{\alpha}(\cdot,0),\tilde{\beta}(\cdot,0))^T\|^2 \nonumber\\
    &\mkern50mu+ \|(\hat{\alpha}(\cdot,0),\hat{\beta}(\cdot,0))^T\|^2)e^{-v^*t}.
\end{align}
Hence, it can be seen that the target systems globally exponentially converge to $0$. Using the bounded invertibility of transformations  \eqref{err_backstepping1},\eqref{err_backstepping2},\eqref{err_backstepping1-inverse},\eqref{err_backstepping2-inverse},\eqref{obs_backstepping_1},\eqref{obs_backstepping_2},\eqref{backstepping_1-inverse}, and \eqref{backstepping_2-inverse}, we can show that the observer-based CETC closed-loop system \eqref{sys_u}-\eqref{sys_BC1},\eqref{obs_u}-\eqref{obs_BC1},\eqref{U_d},\eqref{sys_BC2_d},\eqref{obs_BC2_d} globally exponentially converges to $0$ in the spatial $L^2$ norm. \hfill 
\end{proof}

\begin{remark}
The variables $\theta$ and $\eta$ are free design parameters that can be varied to tune the performance of the controller, in contrast to the CETC design in \cite{espitiaObserverbased2020}, where these parameters are restricted. Also, as evident by the proof of Proposition~\ref{propo:lyapunov stability}, larger $\eta$, larger $\mu$, and smaller $\delta$ would result in a higher convergence rate.
\end{remark}

\section{Periodic event-triggered control (PETC) }\label{sec:petc}


Unlike the CETC approach, in PETC, the triggering function $\Gamma^p(t)$ to be designed is only evaluated periodically. An increasing sequence of PETC times, $I^p=\{t_k^p\}_{k\in\mathbb{N}}$, at which the control input $U_k^p(t)$ is updated, is determined according to the following rule:
\begin{definition}\label{defn:triggering_petc}
The increasing sequence of event times $I^p=\{t_k^p\}_{k\in\mathbb{N}}$ with $t_0^p=0$ is determined according to the following rule:
\begin{align}
    t_{k+1}^p=&\inf\{t\in \mathbb{R}|t>t_k^p,\Gamma^p(t)>0,k\in\mathbb{N},\nonumber\\
    &\mkern120mut=nh, n\in\mathbb{N}, h>0\},\label{tpk}\\
    \Gamma^p(t):=&\left(e^{ah}(\theta_m+a\theta)-\theta_m \right)d^2(t)+am(t),\label{gamma_p}
\end{align}
where $h>0$ is an appropriate sampling period to be chosen. Let $\eta,\theta>0$  be arbitrary parameters. The constants $a$ and $\theta_m$ are chosen as in \eqref{aaaa} and \eqref{thetam} respectively. The function $d(t)$ is given by \eqref{d}, and $m(t)$ evolves according to \eqref{m} with $m(0)=m^0<0,~m(t_k^{p^-})=m(t_k^p)=m(t_k^{p^+})~k\in\mathbb{N}$.
\end{definition}
\subsection{Sampling period selection and design of the triggering function}
Let us first focus on the selection of the sampling period $h>0$. Assume that triggering according to the PETC triggering rule \eqref{tpk},\eqref{gamma_p} ensures that the CETC triggering function $\Gamma^c(t)$ given by \eqref{gamma_c} satisfies $\Gamma^c(t)\leq 0$ for all $t\in\mathbb{R}^+$ along the PETC closed-loop solution. Then, similarly from Lemma \ref{lem:m<0}, it holds that $m(t)$ governed by \eqref{m} satisfies $m(t)<0$ for all $t\in\mathbb{R}^+$. Let an event be triggered at $t=t^p_k$, then from \eqref{gamma_c}, we see that $\Gamma^c(t_k^p)=m(t_k^p)<0$. Then, due to the existence of a minimum dwell-time, we know that $\Gamma^c(t)$ remains negative at least until $t=t_k^p + \tau$. Therefore, let us select the sampling period $h$ as 
\begin{align}0<h\leq \tau.\label{hsmpl}\end{align}

Let us now focus on the design of the periodic event-triggering function $\Gamma^p(t)$. 

\begin{propo}\label{prop:gamma_c_eq} Consider the set of increasing event times $I^p=\{t_k^p\}_{k\in\mathbb{N}}$ with $t_0^p=0$ generated by the PETC triggering rule \eqref{tpk},\eqref{gamma_p} with the sampling period $h>0$ chosen as in \eqref{hsmpl}. Then, the CETC triggering function $\Gamma^c(t)$ given by \eqref{gamma_c} satisfies the following relation
    \begin{align}
        \Gamma^c(t)\leq& \frac{e^{-\eta(t-nh)}}{a}\Bigg[\left(e^{a(t-nh)}(\theta_m+a\theta)-\theta_m\right)d^2(nh)\nonumber\\
        &\qquad+am(nh)\Bigg],\label{gamma_c petc}
    \end{align}
    for all $t\in[nh,(n+1)h),~n\in [t_k^p/h,t_{k+1}^p/h)\cap \mathbb{N},k\in\mathbb{N}$ where $\eta,\theta_m>0$, and $\theta,\kappa_i$, $i=0,1,2$ satisfy \eqref{eqkappa_i}, and $a$ is given by \eqref{aaaa}.
\end{propo}

    \begin{figure*}[b]

        \begin{align}
        \mathbf{B}(t)=\begin{pmatrix}\begin{aligned}&\left(\theta \epsilon_0-\kappa_0\right)\|(\hat{\alpha}(\cdot,t),\hat{\beta}(\cdot,t))^T\|^2
         +\left(\theta \epsilon_1-\kappa_1\right)\hat{\alpha}^2(\ell,t)+\left(\theta \epsilon_2-\kappa_2\right)\tilde{\beta}^2(0,t)-\iota(t)\\
       &\kappa_0 \|(\hat{\alpha}(\cdot,t),\hat{\beta}(\cdot,t))^T\|^2
       -\kappa_1\hat{\alpha}^2(\ell,t)-\kappa_2\tilde{\beta}^2(0,t)\end{aligned}\end{pmatrix}.\label{B}
       \end{align}
    
    \end{figure*} 
\begin{proof}
    Consider a time interval $t\in[nh,(n+1)h),~n\in [t_k^p/h,t_{k+1}^p/h)\cap \mathbb{N},k\in\mathbb{N}$, note that $d(t)$, $m(t)$, and $\Gamma^c(t)$ are continuous for all $t\in(nh,(n+1)h)$. Taking the time derivative of \eqref{gamma_c} for $t\in(nh,(n+1)h)$, the following ODE holds:
    \begin{align*}
        \dot{\Gamma}^c(t)=2\theta d(t)\dot{d}(t)+\dot{m}(t).
    \end{align*}
    Using Young's inequality and \eqref{ddot2},\eqref{m} we can derive the following relation: 
    \begin{align*}
        \dot{\Gamma}^c(t)\leq&\left(1+\epsilon_3+\frac{\theta_m}{\theta}\right)\theta d^2(t) - \eta m(t)\nonumber\\
        &+\left(\theta \epsilon_0-\kappa_0\right)\|(\hat{\alpha}(\cdot,t),\hat{\beta}(\cdot,t))^T\|^2\nonumber\\
        &+\left(\theta \epsilon_1-\kappa_1\right)\hat{\alpha}^2(\ell,t)
        +\left(\theta \epsilon_2-\kappa_2\right)\tilde{\beta}^2(0,t).
    \end{align*}
    Substituting for $d^2(t)$ using \eqref{gamma_c} and introducing $\iota(t)\geq 0$ we arrive at the following ODE:
    \begin{align}
        \dot{\Gamma}^c(t)=&\left(1+\epsilon_3+\frac{\theta_m}{\theta}\right)\Gamma^c(t)-\left(a+\frac{\theta_m}{\theta}\right)m(t)\nonumber\\
         &-\iota(t)+\left(\theta\epsilon_0-\kappa_0\right)\|(\hat{\alpha}(\cdot,t),\hat{\beta}(\cdot,t))^T\|^2\nonumber\\
        &+\left(\theta \epsilon_1-\kappa_1\right)\hat{\alpha}^2(\ell,t)
        +\left(\theta \epsilon_2-\kappa_2\right)\tilde{\beta}^2(0,t).\label{dot_gammac}
    \end{align} 
    Using \eqref{gamma_c} and \eqref{m}, we obtain the following ODE
    \begin{align}
        \dot{m}(t)=&- \kappa_0 \|(\hat{\alpha}(\cdot,t),\hat{\beta}(\cdot,t))^T\|^2  -\kappa_1\hat{\alpha}^2(\ell,t)-\kappa_2\tilde{\beta}^2(0,t)\nonumber\\
        &+\frac{\theta_m}{\theta}\Gamma^c(t)-\left(\eta+\frac{\theta_m}{\theta}\right)m(t).\label{m(gammac)}
    \end{align}
   Define the following matrices
    \begin{align*}
        &\mathbf{z}(t)=\begin{pmatrix}\Gamma^c(t)\\m(t)\end{pmatrix},\
        \mathbf{A}=\begin{pmatrix}\left(1+\epsilon_3+\frac{\theta_m}{\theta}\right) & \left(-a-\frac{\theta_m}{\theta}\right)\\\frac{\theta_m}{\theta}&\left(-\eta-\frac{\theta_m}{\theta}\right)\end{pmatrix},
    \end{align*} 
   and let $\mathbf{B}(t)$ be as given in \eqref{B},
    then from \eqref{dot_gammac} and \eqref{m(gammac)} we obtain
    \begin{align}
        \dot{\mathbf{z}}(t)=\mathbf{A}\mathbf{z}(t)+\mathbf{B}(t).\label{zdot}
    \end{align}
    The solution to \eqref{zdot} for all $t\in[nh,(n+1)h),~n\in [t_k^p/h,t_{k+1}^p/h)\cap \mathbb{N},k\in\mathbb{N}$ is
    \begin{align*}
        \mathbf{z}(t)=e^{\mathbf{A}(t-nh)}\mathbf{z}(nh)+\int_{nh}^te^{\mathbf{A}(t-\xi)}\mathbf{B}(\xi)d\xi.
     \end{align*}
    Let $\mathbf{C}=(1\ 0)$. Since $\Gamma^c(t)=\mathbf{C}\mathbf{z}(t)$, $\Gamma^c(t)$ can be calculated as
    \begin{align*}
        \Gamma^c(t)=\mathbf{C}e^{\mathbf{A}(t-nh)}\mathbf{z}(nh)+\int_{nh}^t\mathbf{C}e^{\mathbf{A}(t-\xi)}\mathbf{B}(\xi)d\xi.
    \end{align*}
    Finding the eigenvalues and eigenvectors of $\mathbf{A}$, $e^{\mathbf{A}(t)}$ can be written as
    \begin{align*}
        e^{\mathbf{A}t}=&\begin{pmatrix}1 & \left(a +\frac{\theta_m}{\theta}\right)\\1 &\frac{\theta_m}{\theta}\end{pmatrix}
        \begin{pmatrix}e^{-\eta t} & 0 \\ 0 & e^{(1+\epsilon_3)t}\end{pmatrix}
        \begin{pmatrix}1 & \left(a +\frac{\theta_m}{\theta}\right)\\1 &\frac{\theta_m}{\theta}\end{pmatrix}^{-1},\\
        =&\frac{e^{-\eta t}}{a}\begin{pmatrix}\frac{\theta_m}{\theta}(e^{at}-1)+ae^{at}&\left(a+\frac{\theta_m}{\theta}\right)(1-e^{at})\\\frac{\theta_m}{\theta}(e^{at}-1)&\frac{\theta_m}{\theta}(1-e^{at}) + a\end{pmatrix}.
    \end{align*}
    Then for $\mathbf{C}e^{\mathbf{A}(t-\xi)}\mathbf{B}(\xi)$ we have
    \begin{align*}
        &\mathbf{C}e^{\mathbf{A}(t-\xi)}\mathbf{B}(\xi)=\frac{e^{-\eta(t-\xi)}}{a}\Bigg[- \frac{g(t-\xi)}{\theta }e^{-\eta (t-\xi)} \iota (\xi)\nonumber\\
        &(g(t-\xi)\epsilon_0-\kappa_0a)\|(\hat{\alpha}(\cdot,\xi),\hat{\beta}(\cdot,\xi))^T\|^2\nonumber\\
        &+ (g(t-\xi)\epsilon_1-\kappa_1a) \hat{\alpha}^2(\ell,\xi) + (g(t-\xi)\epsilon_2-\kappa_2a)\tilde{\beta}^2(0,\xi)\Bigg],
    \end{align*}
    where the increasing function $g(t)>0$ is given by 
    \begin{align*}
        g(t)=(e^{at}(\theta_m+a\theta)-\theta_m).
    \end{align*}
    Now, for $nh\leq\xi\leq t\leq (n+1)h$, using \eqref{eqkappa_i} and \eqref{tau}, $g(t-\xi)\epsilon_i-\kappa_ia$ for $i=0,1,2$ can be shown to satisfy
    \begin{align*}
        g(t-&\xi)\epsilon_i-\kappa_ia\leq g(h)\epsilon_i-\kappa_ia,\\
        \leq&\epsilon_i\left[e^{ah}\left(\theta_m+a\theta\right)-\left(\theta_m+\frac{a\theta}{1-\sigma}\right)\right],\\
        \leq&\epsilon_i(\theta_m+\theta a)(e^{ah}-e^{a\tau}).
    \end{align*}
   Since $h\leq \tau$, $(e^{ah}-e^{a\tau})\leq 0$. Thus, it can be seen that $g(t-\xi)\epsilon_i-\kappa_ia\leq  0$ for $i=0,1,2$. Therefore, we have that $\mathbf{C}e^{\mathbf{A}(t-\xi)}\mathbf{B}(\xi) \leq 0$ leading to the following inequality:
    \begin{align*}
        &\Gamma^c(t)\leq \mathbf{C}e^{\mathbf{A}(t-nh)}\mathbf{z}(nh)\\
        &\leq \frac{e^{-\eta(t-nh)}}{a}\Bigg[\left(\frac{\theta_m}{\theta}(e^{a(t-nh)}-1)+ae^{a(t-nh)}\right)\Gamma^c(nh)\nonumber\\
        &+\left(a+\frac{\theta_m}{\theta}\right)(1-e^{a(t-nh)})m(nh)\Bigg].
    \end{align*}
    Substituting for $\Gamma^c(nh)$ using \eqref{gamma_c}, we obtain \eqref{gamma_c petc}, which concludes the proof. \hfill 
\end{proof}

Finally, inspired by the relation \eqref{gamma_c petc}, let us define the periodic event-triggering function $\Gamma^p(t)$ as in \eqref{gamma_p}.

\subsection{Convergence of the closed-loop system under PETC} 
Using the PETC triggering rule \eqref{tpk},\eqref{gamma_p}, we establish the global exponential convergence of the observer-based PETC closed-loop system \eqref{sys_u}-\eqref{sys_BC1},\eqref{sys_BC2_d},\eqref{obs_u}-\eqref{obs_BC1},\eqref{obs_BC2_d} with the control input $U_k^p(t)$ given by \eqref{U_d}.
\begin{lemm}\label{lem:pect_gammac_m} Under the PETC triggering rule \eqref{tpk},\eqref{gamma_p}  with a sample period $h$ chosen as in \eqref{hsmpl}, the CETC triggering function $\Gamma^c(t)$ given by \eqref{gamma_c} and the dynamic variable $m(t)$ governed by \eqref{m} satisfy $\Gamma^c(t)\leq 0$ and $m(t)<0$ for all $t\in \mathbb{R}^+$.    
\end{lemm}
\begin{proof} 
        Assume that two successive events are triggered at $t=t_k^p$ and $t=t_{k+1}^p$, $k\in\mathbb{N}$ under the PETC triggering rule \eqref{tpk},\eqref{gamma_p} with a sampling period of $h$ and $m(t_k^p)<0$. Since $d(t_k^p)=0$, we have from \eqref{gamma_c} that $\Gamma^c(t_k^p)=m(t_k^p)<0$ and from \eqref{gamma_p} that $\Gamma^p(t_k^p)=am(t_k^p)<0$. Due to the periodic nature of the triggering mechanism, we evaluate $\Gamma^p(t)$ at each $t=nh,~n\in [t_k^p/h,t_{k+1}^p/h)\cap \mathbb{N}$ and events are triggered only if $\Gamma^p(nh)>0$. Consider the inequality \eqref{gamma_c petc} when $t\in[nh,(n+1)h)$. Since $e^{a(t-nh)}$ is an increasing function of t, we have
    \begin{align*}
        \Gamma^c(t)\leq&  \frac{e^{-\eta(t-nh)}}{a}\bigg[\left(e^{ah}\left(\theta_m+a\theta\right)-\theta_m\right)d^2(nh)\nonumber\\
        &\qquad+am(nh)\bigg],\\
        \leq&\frac{e^{-\eta(t-nh)}}{a}\Gamma^p(nh). 
    \end{align*} 
    Therefore,  $\Gamma^p(nh)\leq 0$ implies that $\Gamma^c(t)\leq 0$ for all $t\in[nh,(n+1)h)$. Hence, if $\Gamma^p(t)\leq 0$ at a certain time, $\Gamma^c(t)$ remains non-positive at least until the next evaluation of the triggering function. Since the next time at which $\Gamma^p(t)>0$ is at $t=t_{k+1}^p$, we know that $\Gamma^c(t)\leq 0$ at least until $t=t_{k+1}^{p^-}$. Therefore, from Lemma~\ref{lem:m<0}, $m(t)<0$ for $t\in[t_k^p,t_{k+1}^p)$ and by definition, $m(t_{k+1}^{p^-})=m(t_{k+1}^p)=m(t_{k+1}^{p^+})$, leading to $m(t_{k+1}^p)<0$. Since an event is triggered at $t=t_{k+1}^p$, $d(t_{k+1}^p)=0$, resulting in $\Gamma^c(t_{k+1}^p)=m(t_{k+1}^p)<0$, $\Gamma^p(t_{k+1}^p)=am(t_{k+1}^p)<0$. Applying this reasoning for all intervals in $I^p$ and noting that $m^0<0$, we can conclude that $\Gamma^c(t)<0$ and $m(t)<0$ for all $t\in\mathbb{R}^+$. \hfill 
\end{proof}

Subsequently, we establish the global exponential convergence of the closed-loop system in the following theorem.
\begin{thm}
Subject to Assumption \ref{assum:reflection}, let $\eta,\theta>0$, $\sigma\in(0,1)$ and the parameters $a$, $\theta_m$ be determined as in \eqref{aaaa},\eqref{thetam}  respectively. Then, under the PETC triggering rule \eqref{tpk},\eqref{gamma_p} with the sampling period $h>0$ chosen as in \eqref{hsmpl}, the observer-based PETC closed-loop system \eqref{sys_u}-\eqref{sys_BC1},\eqref{obs_u}-\eqref{obs_BC1},\eqref{U_d},\eqref{sys_BC2_d},\eqref{obs_BC2_d} has a unique solution $(u,v,\hat{u}, \hat{v})^T \in \mathcal{C}^0(\mathbb{R}^+; L^2((0,\ell);\mathbb{R}^4))$, and the closed-loop system states globally exponentially converge to $0$ in the spatial $L^2$ norm.
\end{thm}
\begin{proof}
    Using Corollary~\ref{cor:existece} and noting that the triggering function is evaluated periodically with a period $h$, which makes the closed-loop system inherently Zeno-free, we can conclude that the closed-loop system \eqref{sys_u}-\eqref{sys_BC1},\eqref{obs_u}-\eqref{obs_BC1},\eqref{U_d},\eqref{sys_BC2_d},\eqref{obs_BC2_d} has a unique solution $(u,v,\hat{u}, \hat{v})^T \in \mathcal{C}^0(\mathbb{R}^+; L^2((0,\ell);\mathbb{R}^4))$. Following Lemma~\ref{lem:pect_gammac_m}, we know that $\Gamma^c(t)\leq 0$ and $m(t)<0$ for all $t\in\mathbb{R}^+$ along the PETC closed-loop solution. Therefore, using the same arguments used in Proposition~\ref{propo:lyapunov stability}, we can conclude that the closed-loop system \eqref{sys_u}-\eqref{sys_BC1},\eqref{obs_u}-\eqref{obs_BC1},\eqref{U_d},\eqref{sys_BC2_d},\eqref{obs_BC2_d}, globally exponentially converges to $0$ in the spatial $L^2$ norm. \hfill 
\end{proof}

\section{Self-triggered control (STC)}\label{sec:stc}
In this section, we propose an observer-based STC approach. Unlike the previously proposed CETC and PETC methods, which require evaluating a triggering function to update the control input, the STC approach determines the next event time at the current event time. It does so by using continuously available measurements and predicting an upper bound of the closed-loop system state. Consequently, STC proactively determines event times, whereas both CETC and PETC reactively determine event times after a certain condition is met.  

Before proceeding with the design, we assume the following regarding the initial data.
\begin{assumption}\label{assum:initial}
    The initial conditions of the observer error system given by \eqref{err_u}-\eqref{err_BC2} satisfy
    \begin{align*}
        \tilde{u}^2(x,0)\leq \phi_u,\\
        \tilde{v}^2(x,0)\leq \phi_v,
    \end{align*}
for all $x\in[0,\ell]$  for some known arbitrary constants $\phi_u, \phi_v>0$. 

\end{assumption}

The increasing sequence of event times $I^s=\{t_k^s\}_{k\in\mathbb{N}}$ at which the control input is updated is determined by the following rule. 
\begin{definition}\label{defn:triggering_stc}
 The increasing sequence of event times $I^s=\{t_k^s\}_{k\in\mathbb{N}}$ with $t_0^s=0$ is determined via to the following rule:
    \begin{align}
        t^s_{k+1}=&t^s_k+G(t_k^s),\label{stc}\\
        \intertext{where}
        G(t)=&\max\{\tau,\bar G(t)\},\label{G}\\
        \bar G(t)=&  \frac{1}{\varrho + \eta }\ln\left(\frac{\theta_m\mathcal{F}(t)-m(t)(\varrho+\eta)}{\mathcal{F}(t)(\theta(\varrho+\eta)+\theta_m)}\right),\label{Gbar}\\
        \mathcal{F}(t)=&r_d\left(2\bar{V}_2(t)+\frac{2r_d\bar{D}e^{\bar{\mu}\frac{\ell}{\lambda_2}}\bar{V}_2(t)+\phi(t)}{\varrho}\right)   ,\label{F}\\
        \bar{V}_2(t)=&\int_0^\ell \left(\frac{\bar{C}}{\lambda_1}\hat{\alpha}^2(x,t)e^{-\bar{\mu}\frac{x}{\lambda_1}} + \frac{\bar{D}}{\lambda_2}e^{\bar{\mu}\frac{x}{\lambda_2}}\hat{\beta}^2(x,t)\right)dx,\label{v2_bar}\\
        \phi(t)=&(2\bar{C}q^2+\bar{P}_{V_2})\phi_0(t),\quad\bar{C}=1,\quad\bar{D}=2q^2\label{phi,Cbar,Dbar}\\
        \bar{\mu}=&
        \frac{2\lambda_1\lambda_2}{\ell(\lambda_1+\lambda_2)}\ln\left(\frac{1}{2|q\rho|}\right),\label{mubar}\\
        \phi_0(t)=&\begin{cases}
            \max\{\rho^2\phi_\alpha,\phi_\beta\}&t\leq\frac{\ell}{\lambda_1}+\frac{\ell}{\lambda_2}\\
            0&t>\frac{\ell}{\lambda_1}+\frac{\ell}{\lambda_2}\\
        \end{cases}\label{phi0}\\
        \bar{P}_{V_2}=&\bar{\delta}\int_0^\ell\left(\frac{\bar{C}}{\lambda_1}e^{-\bar{\mu}\frac{x}{\lambda_1}}\bar{p}_1^2(x)+\frac{\bar{D}}{\lambda_2}e^{\bar{\mu}\frac{x}{\lambda_2}}\bar{p}_2^2(x)\right)dx,\label{pv2}\\
        \varrho=&\bar{\delta}-\bar{\mu} + 2\bar{D}e^{\bar{\mu}\frac{\ell}{\lambda_2}},\label{varrho}\\
        r_d=&\frac{4\max\left\{\int_0^\ell(N^\alpha(\xi))^2d\xi,\int_0^\ell(N^\beta(\xi))^2d\xi\right\}}{\min\left\{\frac{\bar{C}e^{-\bar{\mu}\frac{\ell}{\lambda_1}}}{\lambda_1},\frac{\bar{D}}{\lambda_2}\right\}},\label{rd} 
    \end{align}
   with $\bar\delta>0$ chosen such that $\varrho>0$ and $\phi_\alpha$ and $\phi_\beta$ being known constants. Let $\eta,\theta>0$  be arbitrary parameters. Furthermore $\theta_m$ is given by \eqref{thetam}. The dynamic variable $m(t)$ evolves according to \eqref{m} with $m(0)=m^0<0,~m(t_k^{s^-})=m(t_k^s)=m(t_k^{s^+})~k\in\mathbb{N}$. 
\end{definition}

\subsection{Design of the positive function $G(t)$}

The function $G(t)$ is determined such that updating control inputs at events generated by the STC triggering rule \eqref{stc}-\eqref{rd}  ensures the CETC triggering function $\Gamma^c(t)$ given by \eqref{gamma_c} satisfies $\Gamma^c(t)\leq 0$ for all $t\in\mathbb{R}^+$. Towards this, we first derive upper bounds for the variables $d^2(t)$ and $m(t)$ between two consecutive event times $t_k^s$ and $t_{k+1}^s$.
\begin{lemm}\label{lem:d2m eq}
    For $d(t)$ given by \eqref{d}, and for $m(t)$ satisfying \eqref{m} with $m(0)=m^0<0,~m(t_k^{s^-})=m(t_k^s)=m(t_k^{s^+})~k\in\mathbb{N}$, the following inequalities hold:
    \begin{align}
        d^2(t)<& \mathcal{F}(t_k^s)e^{\varrho(t-t_k^s)},\label{d2}\\
        m(t)<& e^{-\eta(t-t_k^s)}m(t_k^s) \nonumber\\
        &+\frac{\theta_m\mathcal{F}(t_k^s)e^{-\eta(t-t_k^s)}}{\varrho+\eta}\left(e^{(\varrho+\eta)(t-t_k^s)}-1\right),\label{meq}
    \end{align}
    for all $t\in[t_k^s,t_{k+1}^s),k\in\mathbb{N}$, with $\varrho>0$ given by \eqref{varrho}. 
\end{lemm}

\begin{proof}
     Differentiating the function \eqref{v2_bar} with respect to time for $t\in[t_k^s,t_{k+1}^s),~k\in\mathbb{N}$ and using \eqref{obs_target_u}-\eqref{obs_target_BC1},\eqref{obs_target_BC2_d}, the following expression can be obtained:
    \begin{align}
       \hspace{-0.08cm} \dot{\bar{V}}_2(t)=&2\tilde{\beta}(0,t)\int_0^\ell\left(\frac{\bar{C}}{\lambda_1}\bar{p}_1(x)e^{-\bar{\mu}\frac{x}{\lambda_1}}+\frac{\bar{D}}{\lambda_2}\bar{p}_2(x)e^{\bar{\mu}\frac{x}{\lambda_2}}\right)dx\nonumber\\
        &-\bar{\mu}\bar{V}_2(t)-\bar{C}\left(e^{-\bar{\mu}\frac{\ell}{\lambda_1}}\hat{\alpha}^2(\ell,t)-\hat{\alpha}^2(0,t)\right)\nonumber\\
        &+\bar{D}\left(e^{\bar{\mu}\frac{\ell}{\lambda_2}}(\rho\hat{\alpha}(\ell,t)+d(t))^2-\hat{\beta}^2(0,t)\right).
    \end{align}
 Then, applying  Young's inequality, we can obtain the following where $\bar{\delta}>0$ and $\bar{P}_{V_2}$ is given in \eqref{pv2}:
    \begin{align}
        \dot{\bar{V}}_2(t)&\leq(\bar{\delta}-\bar{\mu})\bar{V}_2(t)+2\bar{D}e^{\bar{\mu}\frac{\ell}{\lambda_2}}d^2(t)\nonumber\\
        &+\left(-\bar{C}e^{-\bar{\mu}\frac{\ell}{\lambda_2}}+2\rho^2\bar{D}e^{\bar{\mu}\frac{\ell}{\lambda_1}}\right)\hat{\alpha}^2(\ell,t)+\tilde{\beta}^2(0,t)\bar{P}_{V_2}\nonumber\\
        &+\left(2\bar{C}q^2-\bar{D}\right)\hat{\beta}^2(0,t)+2\bar{C}q^2\tilde{\beta}^2(0,t).
    \end{align}
    Selecting $\bar{\mu}$ as defined in \eqref{mubar} such that $1-4\rho^2q^2e^{\bar{\mu}\left(\frac{\ell}{\lambda_1}+\frac{\ell}{\lambda_2}\right)}=0$ along with the parameters $\bar{C}$, $\bar{D}$ as given in \eqref{phi,Cbar,Dbar}, we obtain the following inequality:
    \begin{align}
        \dot{\bar{V}}_2(t)\leq&(\bar{\delta}-\bar{\mu})\bar{V}_2(t)+2\bar{D}e^{\bar{\mu}\frac{\ell}{\lambda_2}}d^2(t) + \tilde{\beta}^2(0,t)(2\bar{C}q+\bar{P}_{V_2}).\label{vdot_bar 1}
    \end{align}
    To obtain an upper bound for the solution of $\bar{V}_2(t)$ using \eqref{vdot_bar 1}, we derive the following inequalities for $d^2(t)$ and $\tilde{\beta}^2(0,t)$. 
    
    Consider $d(t)$ given by \eqref{d}. Using Young's inequality and Cauchy–Schwarz inequality, the following approximation can be obtained for $d^2(t)$:
    \begin{align}
        d^2(t)\leq& 4\max\left\{\int_0^\ell(N^\alpha(\xi))^2d\xi,\int_0^\ell(N^\beta(\xi))^2d\xi\right\}\nonumber\\
        &\times\big(\|(\hat{\alpha}(\cdot,t),\hat{\beta}(\cdot,t))^T\|^2+\|(\hat{\alpha}(\cdot,t_k^s),\hat{\beta}(\cdot,t_k^s))^T\|^2\big),
    \end{align}
    Noting that 
    \begin{align}
        &\min\left\{\frac{\bar{C}e^{-\bar{\mu}\frac{\ell}{\lambda_1}}}{\lambda_1},\frac{\bar{D}}{\lambda_2}\right\}\|(\hat{\alpha}(\cdot,t),\hat{\beta}(\cdot,t))^T\|^2\leq \bar{V}_2(t),
    \end{align}
    the following holds for $d^2(t)$:
    \begin{align}
        d^2(t)\leq r_d (\bar{V}(t^s_k)+\bar{V}(t))\label{d2_v},
    \end{align}
    where $r_d$ is given in \eqref{rd}.
    Additionally considering the dynamics of the target system \eqref{err_target_u}-\eqref{err_target_BC2}, the characteristic solution of the system can be obtained as
    \begin{align}
        \tilde{\alpha}(x,t)=&
        \begin{cases}
            0&x\leq \lambda_1t\\
            \tilde{\alpha}(x-\lambda_1t,0)&x>\lambda_1t
        \end{cases},\label{tildealpha sol}\\
        \tilde{\beta}(x,t)=&
        \begin{cases}
            \tilde{\beta}(x+\lambda_2t,0)&x\leq \ell-\lambda_2t\\
            \rho\tilde{\alpha}\left(\ell+\frac{\lambda_1}{\lambda_2}(\ell-x)-\lambda_1t,0\right)&x> \ell-\lambda_2t
        \end{cases}.\label{tildebeta sol}
    \end{align}
    From Assumption~\ref{assum:initial}, using the transformations \eqref{err_backstepping1-inverse},\eqref{err_backstepping2-inverse}, Young's inequality, and Cauchy–Schwarz inequality, the following inequality can be obtained:
    \begin{align*}
        \tilde{\alpha}^2(x,0)\leq \phi_\alpha,\\
        \tilde{\beta}^2(x,0)\leq \phi_\beta,
    \end{align*}
    where $\phi_\alpha$ and $\phi_\beta$ are given by
    \begin{align}
        \phi_\alpha=& 3\max_{x\in[0,\ell]}\Bigg\{\phi_u + \phi_ux\int_0^x(R^{uu}(x,\xi))^2\d\xi\nonumber\\
        & +\phi_vx\int_0^x(R^{uv}(x,\xi))^2d\xi\Bigg\},\\
        \phi_\beta=& 3\max_{x\in[0,\ell]}\Bigg\{\phi_v + \phi_ux\int_0^x(R^{vu}(x,\xi))^2\d\xi\nonumber\\&+\phi_vx\int_0^x(R^{vv}(x,\xi))^2d\xi\Bigg\}.
    \end{align}
    From the solution of $\tilde{\alpha}(x,t)$ and $\tilde{\beta}(x,t)$ given by \eqref{tildealpha sol},\eqref{tildebeta sol}, it is clear that for $t\geq\frac{\ell}{\lambda_1}+\frac{\ell}{\lambda_2}$, $\tilde{\alpha}(x,t)=0$ and $\tilde{\beta}(x,t)=0$. Also, $\tilde{\beta}^2(0,t)\leq\max\{\rho^2\phi_\alpha,\phi_\beta\}$  for $t<\frac{\ell}{\lambda_1}+\frac{\ell}{\lambda_2}$. Hence $\tilde{\beta}(0,t)$ is bounded for all $ t\geq 0$ as
    \begin{align}
        \tilde{\beta}^2(0,t)\leq \phi_0(t).\label{beta_t,0 eq}
    \end{align}
 Since the $\phi_0(t)$ given by \eqref{phi0} is a non-increasing function of $t$ and since $t\in[t_k^s,t_{k+1}^s)$, \eqref{beta_t,0 eq} still holds for R.H.S at $t=t_k^s$. Substituting $d^2(t)$ from \eqref{d2_v} and $\tilde{\beta}^2(0,t)$ from \eqref{beta_t,0 eq} to \eqref{vdot_bar 1} and  at $\phi_0(t)=\phi_0(t_k^s)$, we obtain the following expression:
    \begin{align*}
        \dot{\bar{V}}_2(t)\leq&\varrho\bar{V}_2(t)+2r_d\bar{D}e^{\bar{\mu}\frac{\ell}{\lambda_2}}\bar{V}_2(t_k^s) + \phi(t_k^s),
    \end{align*}
    where $\varrho$ and $\phi(t)$ are defined in \eqref{varrho} and \eqref{phi,Cbar,Dbar} respectively.
    Using the comparison principle, one gets the following estimate:
    \begin{align}
        \bar{V}_2(t)\leq& \bar{V}_2(t^s_k)e^{\varrho(t-t_k^s)}\nonumber\\
        &+\frac{2r_d\bar{D}e^{\bar{\mu}\frac{\ell}{\lambda_2}}\bar{V}_2(t_k^s) +\phi(t_k^s)}{\varrho} \left(e^{\varrho(t-t_k^s)}-1\right),\nonumber\\
         \bar{V}_2(t)<& \left(\bar{V}_2(t^s_k)+\frac{2r_d\bar{D}e^{\bar{\mu}\frac{\ell}{\lambda_2}}\bar{V}_2(t_k^s) +\phi(t_k^s)}{\varrho} \right)e^{\varrho(t-t_k^s)}.\label{leq bar v2}
    \end{align}
    Using \eqref{d2_v} and \eqref{leq bar v2}, we derive the estimate \eqref{d2}. Next, using \eqref{m} and \eqref{d2}, the following inequality can be derived:
    \begin{align*}
        \dot{m}(t)<-\eta m(t) +\theta_m\mathcal{F}(t_k^s)e^{\varrho(t-t_k^s)}.
    \end{align*} 
Using the comparison principle, the estimate \eqref{meq} can be obtained. \hfill 
\end{proof}

Consider the time period $t\in[t_k^s,t_{k+1}^s),~k\in\mathbb{N}$. Assume that $m(t_k^s)<0$. Since an event is triggered at $t=t_k^s$, we know that $d(t_k^s)=0$, and then, from \eqref{gamma_c}, we know that $\Gamma^c(t_k^s)=m(t_k^s)<0$. Using \eqref{gamma_c},\eqref{d2},and \eqref{meq}, we obtain 
\begin{align}
    &\Gamma^c(t)<\theta \mathcal{F}(t_k^s)e^{\varrho(t-t_k^s)}+  e^{-\eta(t-t_k^s)}m(t_k^s)\nonumber\\
    &+ \frac{\theta_m\mathcal{F}(t_k^s)e^{-\eta(t-t_k^s)}}{\varrho+\eta}\left(e^{(\varrho+\eta)(t-t_k^s)}-1\right).\label{gamma_c_stc}
\end{align} 
R.H.S of \eqref{gamma_c_stc} is an increasing function of $t$. Assume that there exists a time $t_k^{s^*}>t_k^s$ such that the following expression hold:
\begin{align*}
    &\theta \mathcal{F}(t_k^s)e^{\varrho(t_k^{s^*}-t_k^s)}+  e^{-\eta(t_k^{s^*}-t_k^s)}m(t_k^s)\nonumber\\
    &+ \frac{\theta_m\mathcal{F}(t_k^s)e^{-\eta(t_k^{s^*}-t_k^s)}}{\varrho+\eta}\left(e^{(\varrho+\eta)(t_k^{s^*}-t_k^s)}-1\right)=0.
\end{align*}
Solving for $t_k^{s*}$, we get $t_k^{s*}-t_k^s=\bar{G}(t_k^s)$ where $\bar{G}(t)$ is given by \eqref{Gbar}. Since, $t_k^{s*}-t_k^s$ is not guaranteed to be greater than the minimum dwell-time $\tau$, we define $G(t)$ as in \eqref{G}.

\subsection{Convergence of the closed-loop system under STC}
Under the STC triggering rule \eqref{stc}-\eqref{rd}, we establish global exponential convergence of the observer-based system \eqref{sys_u}-\eqref{sys_BC1},\eqref{sys_BC2_d},\eqref{obs_u}-\eqref{obs_BC1},\eqref{obs_BC2_d} with the control input $U_k^s(t)$ given in \eqref{U_d}. Prior to that, we present the following result that is crucial for proving the main result presented in Theorem \ref{thm3}. 
\begin{lemm}\label{lem:stc_gammac_m} Under the STC triggering rule \eqref{stc}-\eqref{rd},  the CETC triggering function $\Gamma^c(t)$ given by \eqref{gamma_c} and the dynamic variable $m(t)$ governed by \eqref{m} satisfy $\Gamma^c(t)\leq 0$ and $m(t)<0$ for all $t\in \mathbb{R}^+$. 
\end{lemm}
\begin{proof}
    Assume that two successive events are triggered at $t=t_k^s$ and $t=t_{k+1}^s,~k\in\mathbb{N}$ under the STC triggering rule \eqref{stc}-\eqref{rd} and $m(t_k^s)<0$. Since $d(t_k^s)=0$, from \eqref{gamma_c}, $\Gamma^c(t_k^s)=m(t_k^s)<0$. According to the definition of the STC triggering rule, we have that $t_{k+1}^s=t_k^s+\tau$ if $\bar{G}(t_k^s)\leq \tau$ or $t_{k+1}^s=t_k^s+\bar{G}(t_k^s)$ if $\bar{G}(t_k^s)>\tau$. Since $\tau$ defined in \eqref{tau} is the minimum dwell-time, if $t_{k+1}^s=t_k^s+\tau$, we know that $\Gamma^c(t)$ remains non positive until $t=t_{k+1}^{s^-}$. Next consider the instance where $t_{k+1}^s=t_k^s+\bar{G}(t_k^s)$. From \eqref{gamma_c_stc}, for $t=t_{k+1}^s$, $\bar{G}(t_k^s)$ is such that
    \begin{align*}
    &\Gamma^c(t_{k+1}^s)<\theta \mathcal{F}(t_k^s)e^{\varrho\bar{G}(t_k^s)}+  e^{-\eta\bar{G}(t_k^s)}m(t_k^s)\nonumber\\
    &+ \frac{\theta_m\mathcal{F}(t_k^s)e^{-\eta(\bar{G}(t_k^s))}}{\varrho+\eta}\left(e^{(\varrho+\eta)(\bar{G}(t_k^s))}-1\right)=0,
    \end{align*}
    therefore, $\Gamma^c(t_{k+1}^s)<0$, and since the upper bound for $\Gamma^c(t)$ in \eqref{gamma_c_stc} is an increasing function of $t$, $\Gamma^c(t)<0$ for all $t\in[t_k^s,t_{k+1}^s)$. Therefore, for the STC condition given in \eqref{stc}-\eqref{rd}, $\Gamma^c(t)$ remains non-positive for $t\in[t_k^s,t_{k+1}^s)$, then similarly from Lemma~\ref{lem:m<0}, we can conclude that $m(t)<0$ for $t\in[t_k^s,t_{k+1}^s)$. By definition, $m(t_{k+1}^{s^-})=m(t_{k+1}^s)=m(t_{k+1}^{s^+})$ leading to $m(t_{k+1}^s)<0$. Since an event is triggered at $t=t_{k+1}^s$, $d(t_{k+1}^s)=0$, resulting in $\Gamma^c(t_{k+1}^s)=m(t_{k+1}^s)<0$. Applying this reasoning for all intervals in $I^s$ and noting that $m^0<0$, we can conclude that $\Gamma^c(t)<0$ and $m(t)<0$ for all $t\in\mathbb{R}^+$. \hfill 
\end{proof}

Note that in light of the results of Lemma~\ref{lem:stc_gammac_m}, a solution for \eqref{Gbar} exist for all $t\in I^s$. Subsequently, we establish the global exponential convergence of the closed-loop system in the following theorem.
\begin{thm}\label{thm3}
Subject to Assumption \ref{assum:reflection}, let $\eta,\theta>0$, $\sigma\in(0,1)$ and the parameters $\tau$, $\theta_m$  be defined as in \eqref{tau},\eqref{thetam} respectively. Then, under the STC triggering rule \eqref{stc}-\eqref{rd}, the observer-based STC closed-loop system \eqref{sys_u}-\eqref{sys_BC1},\eqref{obs_u}-\eqref{obs_BC1},\eqref{U_d},\eqref{sys_BC2_d},\eqref{obs_BC2_d} has a unique solution $(u,v,\hat{u}, \hat{v})^T \in \mathcal{C}^0(\mathbb{R}^+; L^2((0,\ell);\mathbb{R}^4))$, and the closed-loop system states exponentially converge to $0$ in the spatial $L^2$ norm.
\end{thm}
\begin{proof}
    Using Corollary~\ref{cor:existece} and noting that according to the triggering rule \eqref{stc}-\eqref{rd} the time between two events is at least $\tau$, which excludes Zeno behavior, we can conclude that the system \eqref{sys_u}-\eqref{sys_BC1},\eqref{obs_u}-\eqref{obs_BC1},\eqref{U_d},\eqref{sys_BC2_d},\eqref{obs_BC2_d} has a unique solution $(u,v,\hat 
    {u}, \hat{v})^T \in \mathcal{C}^0(\mathbb{R}^+; L^2((0,\ell);\mathbb{R}^4))$. Following Lemma~\ref{lem:stc_gammac_m}, we know that $\Gamma^c(t)\leq0$ and $m(t)<0$ for all $t\in\mathbb{R}^+$. Therefore, using the same arguments used in Proposition~\ref{propo:lyapunov stability}, we can conclude that the closed-loop system \eqref{sys_u}-\eqref{sys_BC1},\eqref{obs_u}-\eqref{obs_BC1},\eqref{U_d},\eqref{sys_BC2_d},\eqref{obs_BC2_d}, exponentially converges to $0$ in the spatial $L^2$ norm. \hfill 
\end{proof}

\section{Boundary control of shallow water wave equations}\label{sec:model}
\emph{\textbf{Motivation.}} Remote hydraulic, pneumatic, or electric  actuation of water gates offers greater flexibility, convenience, and efficiency for the management of water flow  in large-scale irrigation systems or flood control networks\cite{saragih2020design,ushkov2023industrial,telaumbanua2023irrigation, siddula2018water}.  The integration of  remote real-time control stations, acting as a central hub   prevents water wastage and enhances other global performance measures for water distribution systems \cite{prodan2017distributed}.  We present ETC strategies for a single breach water canal that can be expanded to the remote control of networks of irrigation canals \cite{leugering2002modelling}. We consider a single canal operated by a control center located upstream, which uses the measured water height and velocity at the upstream sluice gate to control the downstream gate remotely as given in Fig.~\ref{fig:sv}.

\subsection{The Saint-Venant equations}
Firstly presented  in 1871 \cite{de1871theorie}, the equations describing shallow water dynamics in open canals hold an important place in hydraulic engineering, providing a mathematical representation that is  useful for simulating unsteady flow in open canals and  rivers. Additionally, these Saint-Venant equations have played a key role in the development of boundary controllers for coupled hyperbolic PDE systems.  The Saint-Venant model for a canal breach of unit width, which  expresses the conservation of  mass and momentum assuming that the depth of the incompressible  flow is much smaller than the horizontal length scale, is given below: 
\begin{figure}[t]
\centering
\includegraphics[width=0.45\textwidth]{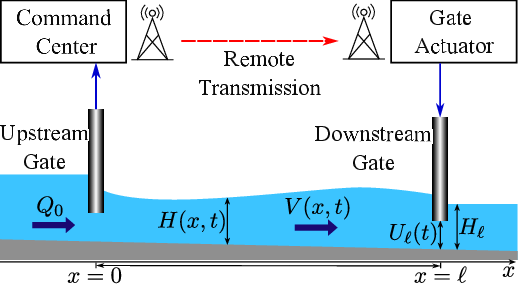}
\caption{Structure of the water canal.} \label{fig:sv}
\end{figure}

\begin{align}
  &\partial_t H+ \partial_x (HV) = 0,\label{sys_Height}\\
  &\partial_t V +\partial_x\bigg(\frac{V^2}{2}+gH\bigg) + \bigg( \frac{C_fV^{2} }{H} - gS_b\bigg) = 0,\label{sys_Velocity}
\end{align}
where $H(x,t)$ represents the water depth,  $V(x,t)$ is the horizontal water velocity, and $g$ is the constant acceleration of gravity. The term $ S_b $ is the constant bottom slope of the channel bed. The constant friction coefficient $C_f$ plays a central role in modulating the fluid's behavior under gravitational influence and resistance posed by the channel (see the extensive discussion of the possible equilibrium in \cite{HAYAT201952}). The constant equilibrium states of $H(x,t)$ and $V(x,t)$ are denoted by $H_{eq}$ and $V_{eq}$, respectively. Assume  that the flow rate upstream is a known constant $Q_0$, and an underflow sluice gate is used at the downstream boundary \cite{prabhata1992}, hence we obtain the following boundary conditions:
\begin{align}
    H(0,t)V(0,t)=&Q_0,\label{sys_boundary1}\\
    H(\ell,t)V(\ell,t)=&k_G\sqrt{2g}U_\ell(t)\sqrt{H(\ell,t)-H_\ell},\label{sys_boundary2}
\end{align}
where $k_G$ represents the constant discharge coefficient of the gate, $H_\ell$ is the constant water level beyond the gate, and $U_\ell$ is the gate opening height which can be controlled to regulate water dynamics in the canal. For the physical stationary states of interest, it is assumed that both $H_{eq}$ and $V_{eq}$ are positive. Let the following assumptions hold for the system.
\begin{assumption}\label{assum:slope}
    The constant bottom slope $S_b$ and the constant equilibrium states $H_{eq}$, $V_{eq}$ are such that
    \begin{align*}
        S_b=\frac{C_fV_{eq}^{2} }{gH_{eq}}.
    \end{align*}
\end{assumption}

\begin{assumption}\label{assum:subcritical}
    The steady states of the system, particularly in scenarios involving navigable rivers and fluvial regimes, the following subcritical condition holds:
    \begin{align*}
        gH_{eq}>V_{eq}^2.
    \end{align*}
\end{assumption}


Then, defining $\tilde H(x,t)=H(x,t)-H_{eq}$ and $\tilde V(x,t)=V(x,t)-V_{eq}$, and linearizing the equations \eqref{sys_Height},\eqref{sys_Velocity} around the steady-state, we obtain the following system:
\begin{align}
\begin{pmatrix}
\tilde H \\
\tilde V
\end{pmatrix}_{t}
+
\begin{pmatrix}
V_{eq} & H_{eq} \\
g & V_{eq}
\end{pmatrix}
\begin{pmatrix}
\tilde H \\
\tilde V
\end{pmatrix}
_{x}+
\begin{pmatrix}
0 & 0 \\
{f_H} & f_V
\end{pmatrix}
\begin{pmatrix}
\tilde H \\
\tilde V
\end{pmatrix}
= \mathbf{0}, \label{equ-uv-error}
\end{align}
where \( f_V \) and \( f_H \) are defined as
\begin{align*}
f_H := -\frac{C_fV_{eq}^{2}}{H_{eq}^{2}}, \quad f_V := \frac{2C_fV_{eq}}{H_{eq}}. 
\end{align*}
Let us denote 
\begin{align*}
M
=
\begin{pmatrix}
V_{eq} & H_{eq} \\
g & V_{eq}
\end{pmatrix}.
\end{align*}
 The matrix $M$ has two real distinct eigenvalues $\lambda_1$ and $-\lambda_2$, which under the subcritical condition in Assumption~\ref{assum:subcritical}, the following expressions hold:
\begin{align*}
\lambda_1 = V_{eq} + \sqrt{gH_{eq}}>0, \quad -\lambda_2 = V_{eq} - \sqrt{gH_{eq}}<0.
\end{align*}
We define the characteristic coordinates as follows
\begin{align}\label{equ-trans-xi}
\begin{pmatrix}
\xi_1 \\
\xi_2
\end{pmatrix}
=
\begin{pmatrix}
\sqrt{\frac{g}{H_{eq}}} & 1 \\
-\sqrt{\frac{g}{H_{eq}}} & 1
\end{pmatrix}
\begin{pmatrix}
\tilde H \\
\tilde V
\end{pmatrix}. 
\end{align}
Linearizing the boundary conditions \eqref{sys_boundary1},\eqref{sys_boundary2}  around the equilibrium point, considering \eqref{equ-uv-error}, and using the transformation \eqref{equ-trans-xi}, we can obtain the following system:
\begin{align}
&\begin{pmatrix}
\xi_1 \\
\xi_2
\end{pmatrix}_t
+
\begin{pmatrix}
\lambda_1 & 0 \\
0 & -\lambda_2
\end{pmatrix}
\begin{pmatrix}
\xi_1 \\
\xi_2
\end{pmatrix}_x
+
\begin{pmatrix}
\gamma_1 & \gamma_2 \\
\gamma_1 & \gamma_2
\end{pmatrix}
\begin{pmatrix}
\xi_1 \\
\xi_2
\end{pmatrix}
= {0}, \label{equ-xi}\\
&\xi_1(0,t) = \tilde{q} \xi_2(0,t),\label{equ0-xi}\\
&\xi_2(\ell,t) = \tilde{\rho} \xi_1(\ell,t) + \tilde{U}(t)\label{equl-xi}.
\end{align}
where $\gamma_1$, $\gamma_2$, $\tilde{q}$, $\tilde{\rho}$, and $\tilde{U}$ are given in Appendix~\ref{app:sv}. As the diagonal coefficients of the source term in \eqref{equ-xi} may bring complexity to the analysis of the stability, we then make a transformation  to remove the diagonal coefficients. We introduce the following new variables:
\begin{align*}
\begin{pmatrix} u \\ v \end{pmatrix} =\begin{pmatrix} e^{\frac{\gamma_1}{\lambda_1}x} & 0 \\ 0 & e^{-\frac{\gamma_2}{\lambda_2}x} \end{pmatrix} \begin{pmatrix} \xi_1 \\ \xi_2 \end{pmatrix}.
\end{align*}
The system \eqref{equ-xi}-\eqref{equl-xi} can now be expressed identical to \eqref{sys_u}-\eqref{sys_BC2} with $c_1(x)$, $c_2(x)$, $q$, $\rho$, and $U(t)$ as defined in Appendix~\ref{app:sv}. 
 It should be noted that although we have proven global exponential convergence for the CETC, PETC, and STC designs, since we linearize the Saint-Venant model, we can only obtain local exponential convergence for the system given by \eqref{sys_Height}-\eqref{sys_boundary2}.  

\subsection{Numerical Simulations}\label{sec:sim}
The simulations for CTC, CETC, PETC, and STC detailed in Sections~\ref{sec:problem_statement}-\ref{sec:stc} respectively are carried out for the one-dimensional shallow water equations given above. The model parameters are defined as $g=9.81 m/s^2$, $\ell=10 m$,  $C_f=0.2$, $H_{eq}=2 m$, $V_{eq}=1m/s$, and $H_\ell=0.1 m$ such that the boundary conditions satisfy Assumption~\ref{assum:reflection}. In addition, we define $k_G=0.6$. The initial conditions for the system are selected such that $H(x,0)=H_{eq} - \sin(\pi \frac{x}{\ell}) $, $V(x,0)=V_{eq} + 0.5\sin(3\pi\frac{x}{\ell})$ along with the initial conditions of the observer states in the characteristic coordinates as $\hat{u}(x,0)=0$, $\hat{v}(x,0)=0$. The selected initial conditions satisfy the boundary conditions \eqref{sys_BC1},\eqref{sys_BC2_d},\eqref{obs_BC1}, and \eqref{obs_BC2_d} with $U_k^\omega(0)=0$. The bounds in Assumption~\ref{assum:initial} are taken as $\phi_u=8.6872$ and $\phi_v=3.1664$.

Subsequently, consider the selection of the event-triggering parameters. Select $\mu=0.016$ such that \eqref{mu,delta} is satisfied. The tuning parameters $\delta<\mu$, $m^0<0$, $\eta>0$, $\theta>0$, and $\sigma\in(0,1)$ are selected as $0.014$ ,$-1$, $0.001$, $1$, and $0.99$ respectively. Thereafter we select $C=413.4211$ from \eqref{C} and $\bar{\delta}=10^{-4}$ such that $\varrho>0$. We obtain a minimum dwell-time of $\tau=0.13323s$ and the sampling period $0<h\leq\tau$ in the PETC design is selected as $h=0.13s$. The time and spatial step sizes used in PDE discretization are $\Delta t=0.0001 s$ and $\Delta x= 0.05 m$, respectively. The kernel PDEs are solved using the method proposed in \cite{anfinsen2019adaptive}[Appendix~F].

The variation of the $L^2$ norms of the characteristic coordinates $\|(u(\cdot,t),v(\cdot,t))^T\|$,  over time is depicted in Fig.~\ref{fig:norm} for the open-loop (OL) system and under the CTC, CETC, PETC, and STC mechanisms. The corresponding control inputs for the triggering mechanisms are shown in Fig.~\ref{fig:U}.  The dwell-times for CETC and PETC are shown in Fig.~\ref{fig:dwell_etc_petc}, and the dwell-times for STC are shown in Fig.~\ref{fig:dwell_stc}. It is apparent that the dwell-times for STC are much shorter than the dwell-times for CETC and PETC. This is because, unlike CETC and PETC, which trigger events upon the violation of a certain condition, the STC approach proactively computes the next event time by predicting the state evolution. This results in a more conservative sampling schedule. Due to frequent control updates, the closed-loop signals under STC follow a trajectory closer to those under CTC, as seen in Fig.~\ref{fig:norm}. In contrast, CETC and PETC approaches are less conservative in determining the triggering times; therefore, the norms converge to zero over a longer period but with less frequent control updates.

\begin{figure}[t]
\centering
\includegraphics[width=0.45\textwidth]{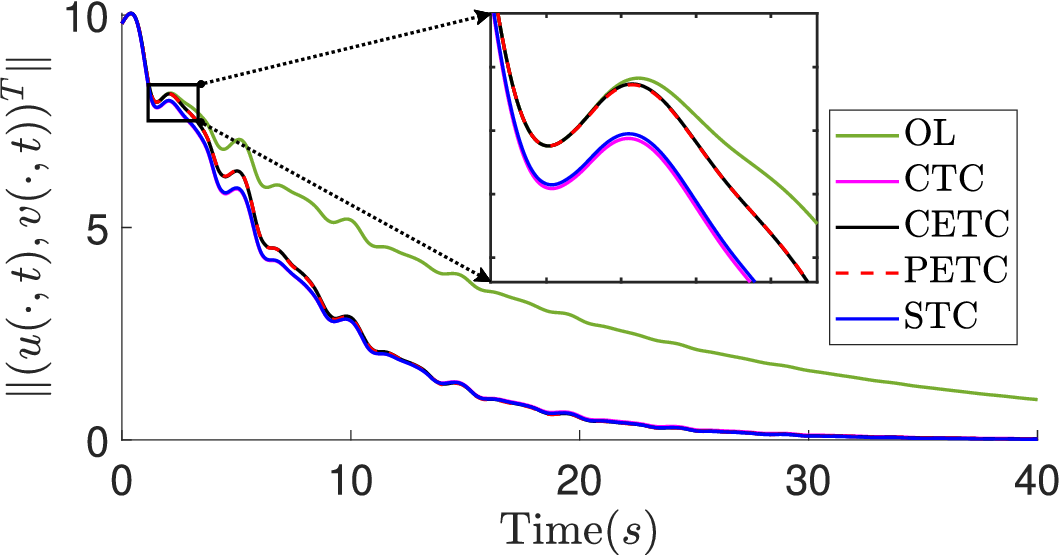}
\caption{Variation of the spatial $L^2$ norm of characteristic coordinates with time.} \label{fig:norm}
\end{figure}

\begin{figure}[t]
\centering
\includegraphics[width=0.45\textwidth]{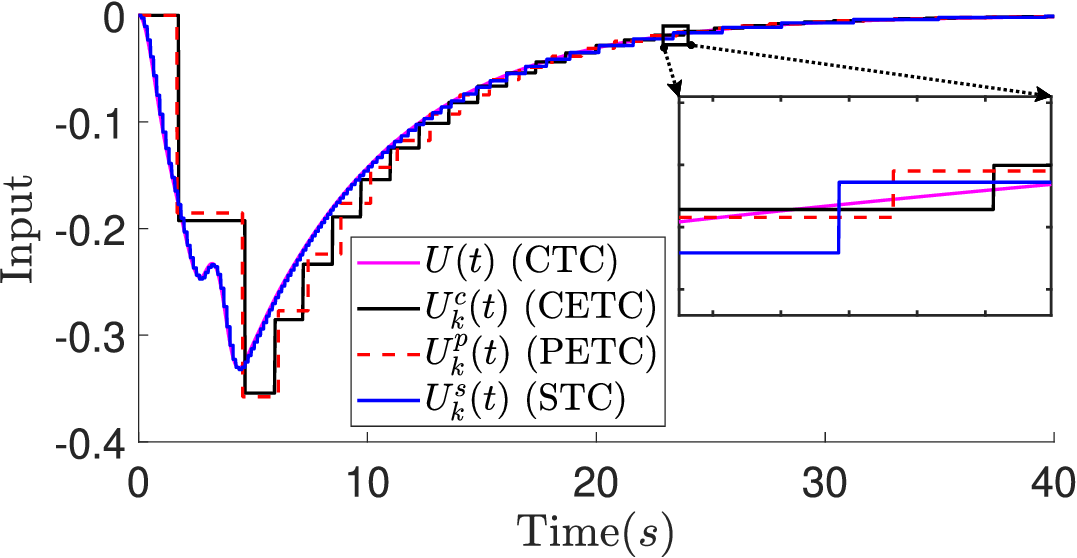}
\caption{Variation of the control inputs with time.} \label{fig:U}
\end{figure}

\begin{figure}[t]
\centering
\includegraphics[width=0.45\textwidth]{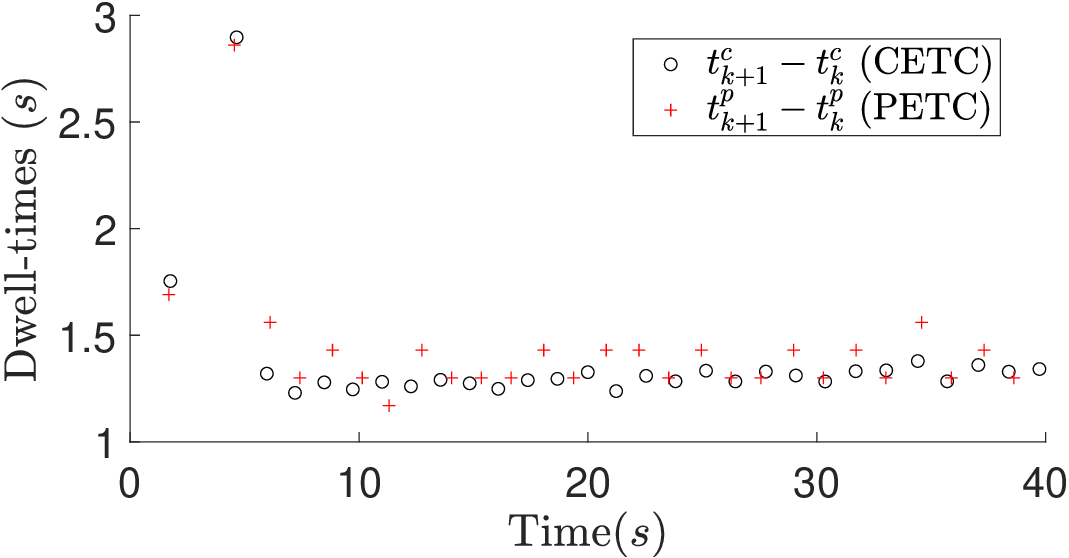}
\caption{Dwell-times under CETC and PETC.} \label{fig:dwell_etc_petc}
\end{figure}

\begin{figure}[t]
\centering
\includegraphics[width=0.45\textwidth]{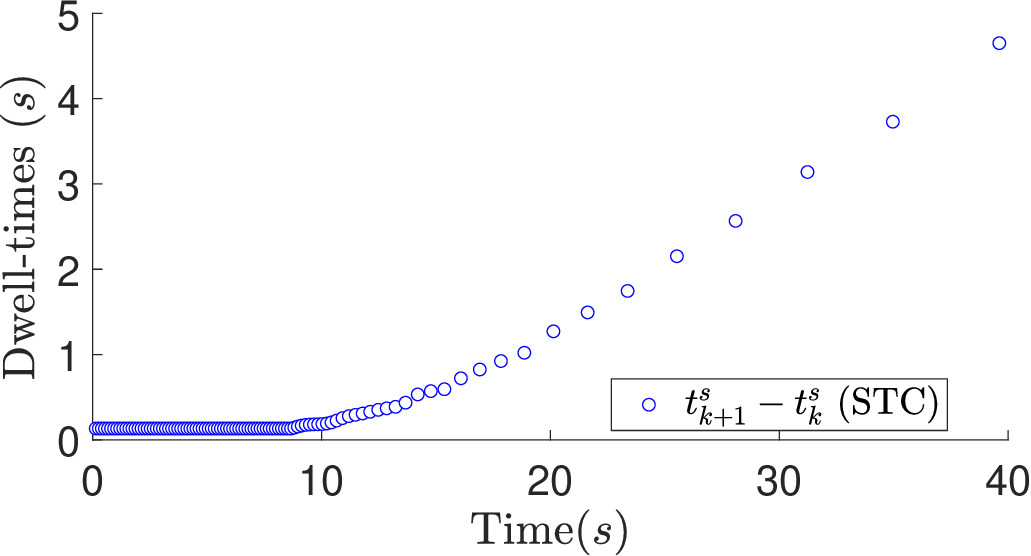}
\caption{Dwell-times under STC.} \label{fig:dwell_stc}
\end{figure}

\section{Concluding remarks}\label{sec:conclusion}
This article has presented an anti-collocated observer-based periodic event-triggered and self-triggered boundary control for a class of $2\times 2$ hyperbolic PDEs with reflection terms at the boundaries. The proposed PETC approach evaluates an appropriately designed function periodically to determine event times and is equipped with an explicitly defined upper bound for the sampling period. In contrast, the STC approach employs a positively lower-bounded function, which is evaluated at an event to determine the next event time. Both approaches eliminate the need for continuous monitoring of a triggering function required in the CETC approach while still preserving global exponential convergence to zero in the spatial $L^2$ norm. We have established the well-posedness of the closed-loop system under the proposed control strategies. The proposed control strategies have been employed to control the linearized Saint-Venant equations, which describe the flow of water in an open channel with a constant inflow of water at the upstream boundary and a sluice gate at the downstream boundary. Numerical simulations illustrating the performance of the proposed CETC, PETC, and STC mechanisms have been provided. Since the control input is transmitted only at event times, the communication bandwidth is utilized effectively and can be freed for other tasks.  This is beneficial in the context of controlling the flow in open channels where a network of canals is remotely controlled by a centralized system. Furthermore, while preserving the rate of  convergence of the system's dynamics,   event-based control strategies gradually open gates, thereby prolonging their life cycle, which is a significant advantage over continuous time control designs.  Our designs can be applied to PDE-ODE cascade models of mining cables, deep-sea construction vessels \cite{wang2022pde}, or oil drilling systems \cite{cai2016nonlinear,cai2020boundary}.

 \appendix
\section{Saint-Venant model parameters}\label{app:sv}
The parameters in \eqref{equ-xi}-\eqref{equl-xi} are defined as
\begin{align*}
    &\gamma_1 = \frac{f_H}{2}\sqrt{\frac{H_{eq}}{g}} + \frac{f_V}{2},\quad
    \gamma_2 = -\frac{f_H}{2}\sqrt{\frac{H_{eq}}{g}} + \frac{f_V}{2},\\
    &\tilde{q}=-\frac{\lambda_2}{\lambda_1},\quad
    \tilde{\rho}=\frac{Q_0-2\lambda_1(H_{eq}-H_\ell)}{Q_0+2\lambda_2(H_{eq}-H_\ell)},\\
    &\rho_u=\frac{4\sqrt{2}gk_G(H_{eq}-H_\ell)^{1.5}}{\sqrt{H_{eq}}(Q_0+2\lambda_2(H_{eq}-H_\ell))},\\
    &U_{eq}=\frac{Q}{k_G\sqrt{2g(H_{eq}-H_\ell)}},\quad
    \tilde{U}(t)=\rho_u(U_\ell(t)-U_{eq}).
\end{align*}
The parameters in \eqref{sys_u}-\eqref{sys_BC2} are defined as shown below to represent the linearized Saint-Venant equations in characteristic coordinates:
\begin{align*}
    &c_1(x)=-\gamma_2e^{\left(\frac{\gamma_1}{\lambda_1}+\frac{\gamma_2}{\lambda_2}\right)x},\quad
    c_2(x)=-\gamma_1e^{-\left(\frac{\gamma_1}{\lambda_1}+\frac{\gamma_2}{\lambda_2}\right)x},\\
    &q=\tilde{q},\quad
    \rho=\tilde{\rho}e^{-\left(\frac{\gamma_1}{\lambda_1}+\frac{\gamma_2}{\lambda_2}\right)\ell},\quad
    U(t)=\tilde{U}e^{\frac{-\gamma_2}{\lambda_2}\ell}.
\end{align*}     

\bibliographystyle{IEEEtranS}
\bibliography{refDataBase}

\begin{thebibliography}{10}
\providecommand{\url}[1]{#1}
\csname url@samestyle\endcsname
\providecommand{\newblock}{\relax}
\providecommand{\bibinfo}[2]{#2}
\providecommand{\BIBentrySTDinterwordspacing}{\spaceskip=0pt\relax}
\providecommand{\BIBentryALTinterwordstretchfactor}{4}
\providecommand{\BIBentryALTinterwordspacing}{\spaceskip=\fontdimen2\font plus
\BIBentryALTinterwordstretchfactor\fontdimen3\font minus \fontdimen4\font\relax}
\providecommand{\BIBforeignlanguage}[2]{{%
\expandafter\ifx\csname l@#1\endcsname\relax
\typeout{** WARNING: IEEEtranS.bst: No hyphenation pattern has been}%
\typeout{** loaded for the language `#1'. Using the pattern for}%
\typeout{** the default language instead.}%
\else
\language=\csname l@#1\endcsname
\fi
#2}}
\providecommand{\BIBdecl}{\relax}
\BIBdecl

\bibitem{aarsnes2014control}
U.~J.~F. Aarsnes, F.~Di~Meglio, S.~Evje, and O.~M. Aamo, ``Control-oriented drift-flux modeling of single and two-phase flow for drilling,'' in \emph{Dynamic Systems and Control Conference}, vol. 46209.\hskip 1em plus 0.5em minus 0.4em\relax American Society of Mechanical Engineers, 2014, p. V003T37A003.

\bibitem{anfinsen2019adaptive}
H.~Anfinsen and O.~M. Aamo, \emph{Adaptive control of hyperbolic {PDE}s}.\hskip 1em plus 0.5em minus 0.4em\relax Springer, 2019.

\bibitem{Anfinsen2016}
H.~Anfinsen, M.~Diagne, O.~M. Aamo, and M.~Krstic, ``An adaptive observer design for $n+1$ coupled linear hyperbolic {PDEs} based on swapping,'' \emph{IEEE Transactions on Automatic Control}, vol.~61, no.~12, pp. 3979--3990, 2016.

\bibitem{anfinsen2017estimation}
------, ``Estimation of boundary parameters in general heterodirectional linear hyperbolic systems,'' \emph{Automatica}, vol.~79, pp. 185--197, 2017.

\bibitem{bastin2010further}
G.~Bastin and J.-M. Coron, ``Further results on boundary feedback stabilisation of 2$\times$ 2 hyperbolic systems over a bounded interval,'' \emph{IFAC Proceedings Volumes}, vol.~43, no.~14, pp. 1081--1085, 2010.

\bibitem{bastin2011boundary}
------, ``On boundary feedback stabilization of non-uniform linear $2\times×2$ hyperbolic systems over a bounded interval,'' \emph{Systems \& Control Letters}, vol.~60, no.~11, pp. 900--906, 2011.

\bibitem{baudouin2023event}
L.~Baudouin, S.~Marx, S.~Tarbouriech, and J.~Valein, ``Event-triggered boundary damping of a linear wave equation,'' \emph{arXiv preprint arXiv:2303.00381}, 2023.

\bibitem{bekiaris20221}
N.~Bekiaris-Liberis, ``On 1-{D} {PDE}-based cardiovascular flow bottleneck modeling and analysis: A vehicular traffic flow-inspired approach,'' \emph{IEEE Transactions on Automatic Control}, 2022.

\bibitem{burkhardt2021stop}
M.~Burkhardt, H.~Yu, and M.~Krstic, ``Stop-and-go suppression in two-class congested traffic,'' \emph{Automatica}, vol. 125, p. 109381, 2021.

\bibitem{cai2020boundary}
X.~Cai and M.~Diagne, ``Boundary control of nonlinear {ODE}/wave {PDE} systems with a spatially varying propagation speed,'' \emph{IEEE Transactions on Automatic Control}, vol.~66, no.~9, pp. 4401--4408, 2020.

\bibitem{cai2016nonlinear}
X.~Cai and M.~Krstic, ``Nonlinear stabilization through wave {PDE} dynamics with a moving uncontrolled boundary,'' \emph{Automatica}, vol.~68, pp. 27--38, 2016.

\bibitem{coron1999lyapunov}
J.-M. Coron, B.~d'Andr{\'e}a Novel, and G.~Bastin, ``A {Lyapunov} approach to control irrigation canals modeled by {Saint-Venant} equations,'' in \emph{1999 European control conference (ECC)}.\hskip 1em plus 0.5em minus 0.4em\relax IEEE, 1999, pp. 3178--3183.

\bibitem{davo2018stability}
M.~A. Dav{\'o}, D.~Bresch-Pietri, C.~Prieur, and F.~Di~Meglio, ``Stability analysis of a 2x2 linear hyperbolic system with a sampled-data controller via backstepping method and looped-functionals,'' \emph{IEEE Transactions on Automatic Control}, vol.~64, no.~4, pp. 1718--1725, 2018.

\bibitem{de1871theorie}
A.~J.-C. de~Saint-Venant \emph{et~al.}, ``Th{\'e}orie du mouvement non-permanent des eaux, avec application aux crues des rivi{\`e}res et {\`a} l’introduction des mar{\'e}es dans leur lit,'' \emph{CR Acad. Sci. Paris}, vol.~73, no. 147-154, pp. 237--240, 1871.

\bibitem{Meglio2013}
F.~Di~Meglio, R.~Vazquez, and M.~Krstic, ``Stabilization of a system of $n+1$ coupled first-order hyperbolic linear {PDEs} with a single boundary input,'' \emph{IEEE Transactions on Automatic Control}, vol.~58, no.~12, pp. 3097--3111, 2013.

\bibitem{di2012backstepping}
F.~Di~Meglio, R.~Vazquez, M.~Krstic, and N.~Petit, ``Backstepping stabilization of an underactuated 3$\times$ 3 linear hyperbolic system of fluid flow equations,'' in \emph{2012 American Control Conference (ACC)}.\hskip 1em plus 0.5em minus 0.4em\relax IEEE, 2012, pp. 3365--3370.

\bibitem{diagne2012lyapunov}
A.~Diagne, G.~Bastin, and J.-M. Coron, ``Lyapunov exponential stability of {1-D} linear hyperbolic systems of balance laws,'' \emph{Automatica}, vol.~48, no.~1, pp. 109--114, 2012.

\bibitem{diagne2017backstepping}
A.~Diagne, M.~Diagne, S.~Tang, and M.~Krstic, ``Backstepping stabilization of the linearized {Saint-Venant-Exner} model,'' \emph{Automatica}, vol.~76, pp. 345--354, 2017.

\bibitem{diagne2021event}
M.~Diagne and I.~Karafyllis, ``Event-triggered boundary control of a continuum model of highly re-entrant manufacturing systems,'' \emph{Automatica}, vol. 134, p. 109902, 2021.

\bibitem{diagne2017control}
M.~Diagne, S.-X. Tang, A.~Diagne, and M.~Krstic, ``Control of shallow waves of two unmixed fluids by backstepping,'' \emph{Annual Reviews in Control}, vol.~44, pp. 211--225, 2017.

\bibitem{espitiaObserverbased2020}
N.~Espitia, ``Observer-based event-triggered boundary control of a linear 2~\texttimes ~2 hyperbolic systems,'' \emph{Systems \& Control Letters}, vol. 138, p. 104668, Apr. 2020.

\bibitem{espitiaTrafficFlowControl2022}
N.~Espitia, J.~Auriol, H.~Yu, and M.~Krstic, ``Traffic flow control on cascaded roads by event-triggered output feedback,'' \emph{International Journal of Robust and Nonlinear Control}, vol.~32, no.~10, pp. 5919--5949, Jul. 2022.

\bibitem{espitia2021event}
N.~Espitia, I.~Karafyllis, and M.~Krstic, ``Event-triggered boundary control of constant-parameter reaction--diffusion {PDEs}: A small-gain approach,'' \emph{Automatica}, vol. 128, p. 109562, 2021.

\bibitem{espitia2020event}
N.~Espitia, H.~Yu, and M.~Krstic, ``Event-triggered varying speed limit control of stop-and-go traffic,'' \emph{IFAC-PapersOnLine}, vol.~53, no.~2, pp. 7509--7514, 2020.

\bibitem{fridman2012robust}
E.~Fridman and A.~Blighovsky, ``Robust sampled-data control of a class of semilinear parabolic systems,'' \emph{Automatica}, vol.~48, no.~5, pp. 826--836, 2012.

\bibitem{2003Boundary}
J.~D. Halleux, C.~Prieur, J.~M. Coron, B.~d\"Andréa Novel, and G.~Bastin, ``Boundary feedback control in networks of open channels,'' \emph{Automatica}, vol.~39, no.~8, pp. 1365--1376, 2003.

\bibitem{hari2021operation}
S.~K.~K. Hari, K.~Sundar, S.~Srinivasan, A.~Zlotnik, and R.~Bent, ``Operation of natural gas pipeline networks with storage under transient flow conditions,'' \emph{IEEE Transactions on Control Systems Technology}, vol.~30, no.~2, pp. 667--679, 2021.

\bibitem{HAYAT201952}
A.~Hayat and P.~Shang, ``A quadratic lyapunov function for {Saint-Venant} equations with arbitrary friction and space-varying slope,'' \emph{Automatica}, vol. 100, pp. 52--60, 2019.

\bibitem{Hu2016}
L.~Hu, F.~Di~Meglio, R.~Vazquez, and M.~Krstic, ``Control of homodirectional and general heterodirectional linear coupled hyperbolic {PDEs},'' \emph{IEEE Transactions on Automatic Control}, vol.~61, no.~11, pp. 3301--3314, 2016.

\bibitem{karafyllis2017sampled}
I.~Karafyllis and M.~Krstic, ``Sampled-data boundary feedback control of {1-D} linear transport {PDEs} with non-local terms,'' \emph{Systems \& Control Letters}, vol. 107, pp. 68--75, 2017.

\bibitem{karafyllis2018sampled}
------, ``Sampled-data boundary feedback control of {1-D} parabolic {PDEs},'' \emph{Automatica}, vol.~87, pp. 226--237, 2018.

\bibitem{katz2022sampled}
R.~Katz and E.~Fridman, ``Sampled-data finite-dimensional boundary control of {1-D} parabolic {PDEs} under point measurement via a novel {ISS} {Halanay’s} inequality,'' \emph{Automatica}, vol. 135, p. 109966, 2022.

\bibitem{katz2020boundary}
R.~Katz, E.~Fridman, and A.~Selivanov, ``Boundary delayed observer-controller design for reaction--diffusion systems,'' \emph{IEEE Transactions on Automatic Control}, vol.~66, no.~1, pp. 275--282, 2020.

\bibitem{koudohode2022event}
F.~Koudohode, L.~Baudouin, and S.~Tarbouriech, ``Event-based control of a damped linear wave equation,'' \emph{Automatica}, vol. 146, p. 110627, 2022.

\bibitem{leugering2002modelling}
G.~Leugering and J.~G. Schmidt, ``On the modelling and stabilization of flows in networks of open canals,'' \emph{SIAM journal on control and optimization}, vol.~41, no.~1, pp. 164--180, 2002.

\bibitem{bojan2022}
B.~Mavkov, T.~Strecker, A.~C. Zecchin, and M.~Cantoni, ``Modeling and control of pipeline networks supplied by automated irrigation channels,'' \emph{Journal of Irrigation and Drainage Engineering}, vol. 148, no.~6, p. 04022015, 2022.

\bibitem{ong2023performance}
P.~Ong and J.~Cort{\'e}s, ``Performance-barrier-based event-triggered control with applications to network systems,'' \emph{IEEE Transactions on Automatic Control}, 2023.

\bibitem{prodan2017distributed}
I.~Prodan, L.~Lefevre, D.~Genon-Catalot \emph{et~al.}, ``Distributed model predictive control of irrigation systems using cooperative controllers,'' \emph{IFAC-PapersOnLine}, vol.~50, no.~1, pp. 6564--6569, 2017.

\bibitem{rathnayake2024observer}
B.~Rathnayake and M.~Diagne, ``Observer-based event-triggered boundary control of the one-phase {Stefan} problem,'' \emph{International Journal of Control}, pp. 1--12, 2024.

\bibitem{rathnayake2023observer}
------, ``Observer-based periodic event-triggered and self-triggered boundary control of a class of parabolic {PDEs},'' \emph{IEEE Transactions on Automatic Control}, vol.~69, no.~12, pp. 8836--8843, 2024.

\bibitem{rathnayake2023prfmnce}
B.~Rathnayake, M.~Diagne, J.~Cortés, and M.~Krstic, ``Performance-barrier event-triggered control of a class of reaction–diffusion {PDEs},'' \emph{Automatica}, vol. 174, p. 112181, 2025.

\bibitem{rathnayake2021observer}
B.~Rathnayake, M.~Diagne, N.~Espitia, and I.~Karafyllis, ``Observer-based event-triggered boundary control of a class of reaction--diffusion {PDEs},'' \emph{IEEE Transactions on Automatic Control}, vol.~67, no.~6, pp. 2905--2917, 2021.

\bibitem{rathnayake2022sampled}
B.~Rathnayake, M.~Diagne, and I.~Karafyllis, ``Sampled-data and event-triggered boundary control of a class of reaction--diffusion {PDEs} with collocated sensing and actuation,'' \emph{Automatica}, vol. 137, p. 110026, 2022.

\bibitem{saragih2020design}
Y.~Saragih, J.~H.~P. Silaban, H.~A. Roostiani, and S.~A. Elisabet, ``Design of automatic water flood control and monitoring systems in reservoirs based on internet of things ({IoT}),'' in \emph{2020 3rd International Conference on Mechanical, Electronics, Computer, and Industrial Technology (MECnIT)}.\hskip 1em plus 0.5em minus 0.4em\relax IEEE, 2020, pp. 30--35.

\bibitem{siddula2018water}
S.~S. Siddula, P.~Babu, and P.~Jain, ``Water level monitoring and management of dams using {IoT},'' in \emph{2018 3rd international conference on internet of things: smart innovation and usages ({IoT-SIU})}.\hskip 1em plus 0.5em minus 0.4em\relax IEEE, 2018, pp. 1--5.

\bibitem{strecker2022}
T.~Strecker, M.~Cantoni, and E.~Weyer, ``Switched control of over-topping water channels,'' \emph{IFAC-PapersOnLine}, vol.~55, no.~33, pp. 78--84, 2022, 2nd IFAC Workshop on Control Methods for Water Resource Systems CMWRS 2022.

\bibitem{prabhata1992}
P.~K. Swamee, ``Sluice-gate discharge equations,'' \emph{Journal of Irrigation and Drainage Engineering}, vol. 118, no.~1, pp. 56--60, 1992.

\bibitem{telaumbanua2023irrigation}
M.~Telaumbanua, E.~K. Baene, R.~Ridwan, A.~Haryanto, F.~K. Wisnu, and S.~Suharyatun, ``Irrigation water gate monitoring system based on the internet of things using microcontroller,'' in \emph{AIP Conference Proceedings}, vol. 2623, no.~1.\hskip 1em plus 0.5em minus 0.4em\relax AIP Publishing, 2023.

\bibitem{ushkov2023industrial}
A.~Ushkov, N.~Strelkov, V.~Krutskikh, and A.~Chernikov, ``Industrial internet of things platform for water resource monitoring,'' in \emph{2023 international Russian smart industry conference (SmartIndustryCon)}.\hskip 1em plus 0.5em minus 0.4em\relax IEEE, 2023, pp. 593--599.

\bibitem{Vazquez2011}
R.~Vazquez, M.~Krstic, and J.-M. Coron, ``Backstepping boundary stabilization and state estimation of a $2\times 2$ linear hyperbolic system,'' in \emph{2011 50th IEEE Conference on Decision and Control and European Control Conference}, 2011, pp. 4937--4942.

\bibitem{wang2022event}
J.~Wang and M.~Krstic, ``Event-triggered adaptive control of a parabolic {PDE-ODE} cascade with piecewise-constant inputs and identification,'' \emph{IEEE Transactions on Automatic Control}, 2022.

\bibitem{wang2022eventb}
------, ``Event-triggered adaptive control of coupled hyperbolic {PDEs} with piecewise-constant inputs and identification,'' \emph{IEEE Transactions on Automatic Control}, vol.~68, no.~3, pp. 1568--1583, 2022.

\bibitem{wang2022pde}
------, \emph{PDE Control of String-Actuated Motion}.\hskip 1em plus 0.5em minus 0.4em\relax Princeton University Press, 2022, vol.~73.

\bibitem{wang2023neural}
S.~Wang, M.~Diagne, and M.~Krsti{\'c}, ``Neural operator approximations of backstepping kernels for $2\times 2$ hyperbolic {PDE}s,'' \emph{arXiv preprint arXiv:2312.16762}, 2023.

\bibitem{wang2022sampled}
X.~Wang, Y.~Tang, C.~Fiter, and L.~Hetel, ``Sampled-data distributed control for homo-directional linear hyperbolic system with spatially sampled state measurements,'' \emph{Automatica}, vol. 139, p. 110183, 2022.

\bibitem{yu2019traffic}
H.~Yu and M.~Krstic, ``Traffic congestion control for {Aw--Rascle--Zhang} model,'' \emph{Automatica}, vol. 100, pp. 38--51, 2019.

\bibitem{peihan2024traffic}
P.~Zhang, B.~Rathnayake, M.~Diagne, and M.~Krstic, ``Performance-barrier event-triggered pde control of traffic flow,'' \emph{IEEE Transactions on Automatic Control}, pp. 1--16, 2025, http://dx.doi.org/10.1109/TAC.2025.3547958.

\end{thebibliography}

\end{document}